\documentclass{amsart}

\usepackage{amssymb}

\newtheorem{theorem}{Theorem}[section]
\newtheorem{lemma}[theorem]{Lemma}
\newtheorem{corollary}[theorem]{Corollary}

\theoremstyle{definition}

\newtheorem{example}[theorem]{Example}

\theoremstyle{remark}
\newtheorem{remark}[theorem]{Remark}

\numberwithin{equation}{section}

\allowdisplaybreaks[4]

\begin{document}

\title{Mixed $L^{p}(L^{2})$ norms of the lattice point discrepancy}


 \author[L. Colzani]{Leonardo Colzani}
\address{Dipartimento di Matematica e Applicazioni,
   Universit\`a degli Studi di Milano-Bicocca,
   Via R. Cozzi 55, 20125 Milano,
   Italy}
\curraddr{}
\email{leonardo.colzani@unimib.it}
\thanks{}

\author[B. Gariboldi]{Bianca Gariboldi}
\address{Dipartimento di Matematica e Applicazioni,
   Universit\`a degli Studi di Milano-Bicocca,
   Via R. Cozzi 55, 20125 Milano,
   Italy}
\curraddr{}
\email{b.gariboldi@campus.unimib.it}
\thanks{}

\author[G. Gigante]{Giacomo Gigante}
\address{Dipartimento di Ingegneria Gestionale, dell'Informazione e della Produzione,
  Universit\`a degli Studi di Bergamo,
  Viale Marconi 5, 24044 Dalmine (BG),
  Italy}
\curraddr{}
\email{giacomo.gigante@unibg.it}
\thanks{}

\date{}

\dedicatory{}

\begin{abstract}
We estimate some mixed $L^{p}\left(  L^{2}\right)  $\ norms of the discrepancy
between the volume and the number of integer points in $r\Omega-x$, a dilated
by a factor $r$\ and translated by a vector $x$\ of a convex body $\Omega$\ in
$\mathbb{R}^{d}$,
\[
\left\{  {\displaystyle\int_{\mathbb{T}^{d}}}\left(  \dfrac{1}{H}%
{\displaystyle\int_{R}^{R+H}}\left\vert \sum_{k\in\mathbb{Z}^{d}}\chi
_{r\Omega-x}(k)-r^{d}\left\vert \Omega\right\vert \right\vert ^{2}dr\right)
^{p/2}dx\right\}  ^{1/p}.
\]
We obtain estimates for fixed values of $H$ and $R\to\infty$,  
and also asymptotic estimates when $H\to\infty$.
\end{abstract}

\maketitle

\section{Introduction}

\label{intro}

The discrepancy between the volume and the number of integer points in
$r\Omega-x$, a dilated by a factor $r$ and translated by a vector $x$ of
bounded domain $\Omega$ in $\mathbb{R}^{d}$, is
\[
\mathcal{D}\left(  r\Omega-x\right)  =%
{\displaystyle\sum_{k\in\mathbb{Z}^{d}}}
\chi_{r\Omega-x}(k)-r^{d}\left\vert \Omega\right\vert .
\]

Here $\chi_{r\Omega-x}(y)$ denotes the characteristic function of $r\Omega-x$
and $\left\vert \Omega\right\vert $ the measure of $\Omega$. 
A classical
problem is to estimate the size of $\mathcal{D}\left(  r\Omega-x\right) $, as
$r\rightarrow+\infty$. For a survey see e.g. \cite{Kratzel} and  \cite{IKKN}. We want to estimate the
mixed $L^{p}\left(  L^{2}\right) $ norms of this discrepancy:
\[
\left\{  {%
{\displaystyle\int_{\mathbb{T}^{d}}}
}\left[  \dfrac{1}{H}{%
{\displaystyle\int_{R}^{R+H}}
}\left\vert \mathcal{D}\left(  r\Omega-x\right)  \right\vert ^{2}dr\right]
^{p/2}dx\right\}  ^{1/p}.
\]


In order to present our results, we need to introduce some notation. If $d\mu(r)$ is a
finite Borel measure on the line $-\infty<r<+\infty$, and if $0<H<+\infty$ and
$-\infty<R<+\infty$, the dilated and translated measure $d\mu_{H,R}\left(
r\right)  $ is defined by
\[
\mu_{H,R}\left\{  I\right\}  =\mu\left\{  H^{-1}\left(  I-R\right)  \right\}
.
\]

Alternatively, by duality with continuous bounded functions,
\[%
{\displaystyle\int_{\mathbb{R}}}
f\left(  r\right)  d\mu_{H,R}\left(  r\right)  =%
{\displaystyle\int_{\mathbb{R}}}
f\left(  R+Hr\right)  d\mu\left(  r\right)  .
\]

With this definition, the Fourier transforms of $d\mu\left(  r\right)  $ and
$d\mu_{H,R}\left(  r\right)  $ are related by the equation
\begin{align*}
&  \widehat{\mu}_{H,R}\left(  \zeta\right)  =%
{\displaystyle\int_{\mathbb{R}}}
\exp\left(  -2\pi i\zeta r\right)  d\mu_{H,R}\left(  r\right) \\
&  =%
{\displaystyle\int_{\mathbb{R}}}
\exp\left(  -2\pi i\zeta\left(  R+Hr\right)  \right)  d\mu\left(  r\right)
=\exp\left(  -2\pi iR\zeta\right)  \widehat{\mu}\left(  H\zeta\right)  .
\end{align*}

Recall that the Fourier dimension of a measure is the supremum of all $\delta$
such that there exists $C$ such that $\left\vert \widehat{\mu}\left(
\zeta\right)  \right\vert \leq C\left\vert \zeta\right\vert ^{-\delta/2}$.
See \cite[Section 4.4]{falconer} and \cite[Section 12.17]{PM}.
By a classical result of D.Kendall,  the $L^{2}$ norm of the discrepancy of an oval $\mathcal{D}\left(  r\Omega
-x\right)  $ is of the order of $r^{\left(  d-1\right)  /2}$. See \cite{Kendall} and what follows. For this reason we shall call
$r^{-\left(  d-1\right)  /2}\mathcal{D}\left(  r\Omega-x\right)  $ the
normalized discrepancy. Our main result below is an estimate of
the Fourier dimension of the set where this normalized discrepancy may be large.

\begin{theorem}
\label{1}\ Assume that $d\mu\left(  r\right)  $ is a Borel probability measure
on $\mathbb{R}$, with support in $\varepsilon<r<\delta$, with $\delta
>\varepsilon>0$, and assume that the Fourier transform of $d\mu\left(
r\right)  $ has the decay
\[
\left\vert \widehat{\mu}\left(  \zeta\right)  \right\vert \leq B\left(
1+\left\vert \zeta\right\vert \right)  ^{-\beta},
\]
for some $\beta\geq0$ and $B>0$. Assume that $\Omega$\ is a convex set in
$\mathbb{R}^{d}$, with a smooth boundary with strictly positive Gaussian
curvature. Finally, for given $d$ and $\beta$, define $A$ and $\alpha$ as follows:
\[
\left\{
\begin{array}
[c]{llll}%
d=2, & 0\leq\beta<1, & A={4}/({1-\beta}), & \alpha=({1+\beta})/{4},\\
d=2, & \beta=1, & A=+\infty, & \alpha=1,\\
d=2, & \beta>1, & A=+\infty, & \alpha=1/2,
\end{array}
\right.
\]%
\[
\left\{
\begin{array}
[c]{llll}%
d=3, & 0\leq\beta\leq1/2, & A=({3-2\beta})/({1-\beta}), & \alpha
=({1-\beta})/({3-2\beta}),\\
d=3, & 1/2\leq\beta<1, & A={6}/({2-\beta}), & \alpha=({1+\beta})/{6},\\
d=3, & \beta=1, & A=6, & \alpha=5/6,\\
d=3, & \beta>1 & A=6, & \alpha=1/3,
\end{array}
\right.
\]%
\[
\left\{
\begin{array}
[c]{llll}%
d\geq4, & 0\leq\beta<1, & A=({2d-4\beta})/({d-1-2\beta}), & \alpha
=({d-1-2\beta})/({2d-4\beta}),\\
d\geq4, & \beta=1, & A=({2d-4})/({d-3}), & \alpha=({d-1})/({2d-4}),\\
d\geq4, & \beta>1, & A=({2d-4})/({d-3}), & \alpha=({d-3})/({2d-4}).
\end{array}
\right.
\]
Then the following hold:

\noindent (1) If $p<A$, then there exists $C$ such that for every $H,R\geq1$,
\[
\left\{  \int_{\mathbb{T}^{d}}\left(
{\displaystyle\int_{\mathbb{R}}}
\left\vert r^{-\left(  d-1\right)  /2}\mathcal{D}\left(  r\Omega-x\right)
\right\vert ^{2}d\mu_{H,R}\left(  r\right)  \right)  ^{p/2}dx\right\}
^{1/p}\leq C\left(  \dfrac{1}{p}-\dfrac{1}{A}\right)  ^{-\alpha}.
\]
(2) If $p=A$, then there exists $C$ such that for every $H,R\geq1$,
\[
\left\{  \int_{\mathbb{T}^{d}}\left(
{\displaystyle\int_{\mathbb{R}}}
\left\vert r^{-\left(  d-1\right)  /2}\mathcal{D}\left(  r\Omega-x\right)
\right\vert ^{2}d\mu_{H,R}\left(  r\right)  \right)  ^{p/2}dx\right\}
^{1/p}\leq C\log^{\alpha}\left(  1+R\right)  .
\]
(3) If $\beta>0$ and $p<A$, the family of functions indexed by $R$ and $H$
\[
\int_{\mathbb{R}}\left\vert r^{-\left(  d-1\right)  /2}\mathcal{D}\left(
r\Omega-x\right)  \right\vert ^{2}d\mu_{H,R}\left(  r\right)
\]
has a limit $\mathcal{G}\left(  x\right)  $ in the norm of $L^{p/2}\left(
\mathbb{T}^{d}\right)  $ as $H\rightarrow+\infty$. In particular, the convergence in norm implies the
convergence of the norms, and this is uniform in $R\geq1$,
\begin{equation*}
 \lim_{H\rightarrow+\infty}\left\{  \int_{\mathbb{T}^{d}}\left(
{\displaystyle\int_{\mathbb{R}}}
\left\vert r^{-\left(  d-1\right)  /2}\mathcal{D}\left(  r\Omega-x\right)
\right\vert ^{2}d\mu_{H,R}\left(  r\right)  \right)  ^{p/2}dx\right\}
^{1/p}
  =\left\{  \int_{\mathbb{T}^{d}}\left\vert \mathcal{G}\left(  x\right)
\right\vert ^{p/2}dx\right\}  ^{1/p}.
\end{equation*}

\end{theorem}

The growth $(1/p-1/A)^{-\alpha}$ of the norm  of the discrepancy in the above theorem allows to extrapolate
some Orlicz type estimates at the critical index $p=A$.

\begin{corollary}
\label{1.1} (1) Assume one of the following rows of indexes:
\[
\left\{
\begin{array}
[c]{llll}%
d=2, & \beta=1, & \alpha=2, & \gamma<2/e,\\
d=2, & \beta>1, & \alpha=1, & \gamma<1/e.
\end{array}
\right.
\]
Then there exists $C>0$\ such that for every $H,R\geq1$,
\[
{%
{\displaystyle\int_{\mathbb{T}^{2}}}
\exp}\left(  \gamma\left(
{\displaystyle\int_{\mathbb{R}}}
\left\vert r^{-\left(  d-1\right)  /2}\mathcal{D}\left(  r\Omega-x\right)
\right\vert ^{2}d\mu_{H,R}\left(  r\right)  \right)  ^{1/\alpha}\right)
dx\leq C.
\]
(2) Assume one of the following rows of indexes:
\[
\left\{
\begin{array}
[c]{llll}%
d=2, & 0\leq\beta<1, & p=4/\left(  1-\beta\right)  , & \gamma>2/\left(
1-\beta\right)  ,
\end{array}
\right.
\]%
\[
\left\{
\begin{array}
[c]{llll}%
d=3, & 0\leq\beta\leq1/2, & p=\left(  3-2\beta\right)  /\left(  1-\beta
\right)  , & \gamma>2,\\
d=3, & 1/2\leq\beta<1, & p=6/\left(  2-\beta\right)  , & \gamma>3/\left(
2-\beta\right)  ,\\
d=3, & \beta=1, & p=6, & \gamma>6,\\
d=3, & \beta>1, & p=6, & \gamma>3,
\end{array}
\right.
\]%
\[
\left\{
\begin{array}
[c]{llll}%
d\geq4, & 0\leq\beta<1, & p=\left(  2d-4\beta\right)  /\left(  d-1-2\beta
\right)  , & \gamma>2,\\
d\geq4, & \beta=1, & p=\left(  2d-4\right)  /\left(  d-3\right)  , &
\gamma>\left(  2d-4\right)  /\left(  d-3\right)  ,\\
d\geq4, & \beta>1, & p=\left(  2d-4\right)  /\left(  d-3\right)  , & \gamma>2.
\end{array}
\right.
\]
Then there exists $C$ such that for every $H,R\geq1$,
\begin{align*}
&  \int_{\mathbb{T}^{d}}\left(
{\displaystyle\int_{\mathbb{R}}}
\left\vert r^{-\left(  d-1\right)  /2}\mathcal{D}\left(  r\Omega-x\right)
\right\vert ^{2}d\mu_{H,R}\left(  r\right)  \right)  ^{p/2}\\
&  \times\log^{-\gamma}\left(  2+%
{\displaystyle\int_{\mathbb{R}}}
\left\vert r^{-\left(  d-1\right)  /2}\mathcal{D}\left(  r\Omega-x\right)
\right\vert ^{2}d\mu_{H,R}\left(  r\right)  \right)  dx\leq C.
\end{align*}

\end{corollary}

The existence of the limit function $\mathcal G (x)$ in the above theorem
is somehow related to the quasi-periodicity of the normalized discrepancy
as a function of the dilation parameter $r$. Indeed,
for \textit{\textquotedblleft generic\textquotedblright} convex sets, the
theorem can be slightly strengthened.

\begin{corollary}
\label{1.2} If the support function of the convex set $\Omega$, defined as $g\left(  x\right)
=\sup_{y\in\Omega}\left\{  x\cdot y\right\}  $, has the property that there
exists $C$ such that for every $m$ in $\mathbb{Z}^{d}$ the equation $g\left(
m\right)  =g\left(  n\right)  $ has at most $C$ solutions $n$ in
$\mathbb{Z}^{d}$, then the limit function $\mathcal{G}\left(  x\right)  $ in
 Theorem \ref{1} (3) is bounded and continuous in $\mathbb{T}^{d}$. If the
support function is injective when restricted to the integers, that is
$g\left(  m\right)  \neq g\left(  n\right)  $ for every $m,n\in\mathbb{Z}^{d}$
with $m\neq n$, then this limit function $\mathcal{G}\left(  x\right)  $ is constant.
\end{corollary}

The family of compact convex sets endowed with the Hausdorff metric is a
complete metric space. It can be shown that the collection of convex sets with
injective support functions is an intersection of a countable family of open
dense sets. In particular, we will observe in Remark \ref{ellissoidi} 
that the above corollary applies to
almost every ellipsoid $\left\{  \left\vert M\left(  x-p\right)  \right\vert
\leq1\right\}  $, but it does not apply to the ball $\left\{  \left\vert x\right\vert
\leq1\right\} $. The ball is different. This has as a consequence an estimate from below 
for the norm of the discrepancy.


\begin{theorem}
\label{2}\ (1) Assume that $\Sigma=\left\{  \left\vert x\right\vert
\leq1\right\}  $\ is the ball in $\mathbb{R}^{d}$ with $d\geq4$, and that
$d\mu\left(  r\right)  $ is an arbitrary Borel probability measure on $\mathbb{R}$. Then
for every $p>2d/\left(  d-3\right)  $,
\[
\limsup_{H,R\rightarrow+\infty}\left\{  \int_{\mathbb{T}^{d}}\left(
{\displaystyle\int_{\mathbb{R}}}
\left\vert r^{-\left(  d-1\right)  /2}\mathcal{D}\left(  r\Sigma-x\right)
\right\vert ^{2}d\mu_{H,R}\left(  r\right)  \right) ^{p/2}dx\right\}
^{1/p}=+\infty.
\]
(2) If in addition the Fourier transform of $d\mu\left(  r\right)  $ vanishes
at infinity, $\lim_{\left\vert \zeta\right\vert \rightarrow+\infty}\left\{
\left\vert \widehat{\mu}\left(  \zeta\right)  \right\vert \right\}  =0$, the
supremum limit can be replaced by a limit,
\[
\lim_{H,R\rightarrow+\infty}\left\{  \int_{\mathbb{T}^{d}}\left(
{\displaystyle\int_{\mathbb{R}}}
\left\vert r^{-\left(  d-1\right)  /2}\mathcal{D}\left(  r\Sigma-x\right)
\right\vert ^{2}d\mu_{H,R}\left(  r\right)  \right)  ^{p/2}dx\right\}
^{1/p}=+\infty.
\]

\end{theorem}

The range $p>2d/\left(  d-3\right)  $ in this theorem should be compared with the range
$p<(2d-4)/\left(  d-3\right)  $ in the previous theorem. What happens in between
 is an open problem.
In any case, observe that both these indexes $\left(  2d-4\right)  /\left(  d-3\right)  $ and
$2d/\left(  d-3\right)  $ are asymptotic to 2 as $d\rightarrow+\infty$.
The proof of this theorem reduces essentially to an estimate of the norm in
$L^{p/2}\left(  \mathbb{T}^{d}\right)  $ of the function $\mathcal{G}\left(
x\right)  $ which appears as a limit of the discrepancy in Theorem \ref{1}.
While the limit function associated to the ball $\left\{  \left\vert
x\right\vert \leq1\right\}  $\ is unbounded, for a generic ellipsoid the limit
function is constant. In particular, we do not know if the statement of the
theorem for the ball also applies to all convex sets.

The techniques used to prove the above theorems also apply to the estimates of
pure $L^{p}$ norms of the discrepancy:
\[
\left\{  {%
{\displaystyle\int_{\mathbb{T}^{d}}}
\int_{\mathbb{R}}}\left\vert r^{-\left(  d-1\right)  /2}\mathcal{D}\left(
r\Omega-x\right)  \right\vert ^{p}d\mu_{H,R}\left(  r\right)  dx\right\}
^{1/p}.
\]

This will be addressed in another paper. See also \cite{Gariboldi}.

\bigskip
Finally, we would like to thank our friends Luca Brandolini and Giancarlo Travaglini
for continuous discussions during the preparation of this paper. 

\section{Previous results} \label{previous}
In order to put in an appropriate perspective our results, let us
present a short non-exhaustive list of previous results on the discrepancy.

Studying the arithmetical function $r(n)$, the number of integer pairs $(h,k)$
with $n=h^{2}+k^{2}$, G.H.Hardy in \cite{Hardy_sotto} and E.Landau in \cite{Landau} proved that the
discrepancy $\left\vert
{\sum_{n\leq T}}
r(n)-\pi T\right\vert $ can be larger than $CT^{1/4}\log^{1/4}\left(
T\right)  $. On the other hand, in \cite{Hardy2} Hardy proved that the
mean square average of this discrepancy satisfies for every $\varepsilon>0$ the
estimate
\[
\left\{  \dfrac{1}{T}%
{\displaystyle\int_{0}^{T}}
\left\vert
{\displaystyle\sum_{n\leq t}}
r(n)-\pi t\right\vert ^{2}dt\right\}  ^{1/2}\leq CT^{1/4+\varepsilon}.
\]

In our notation $\sum_{n\leq t}r(n)-\pi t$ is nothing but the discrepancy
$\mathcal{D}\left(  \sqrt{t}\Omega\right)  $ of the disc $\Omega=\left\{
\left\vert x\right\vert \leq1\right\}  $ in the plane. Hardy also stated that
it is not unlikely that the supremum norm of this discrepancy is dominated by
$T^{1/4+\varepsilon}$ for every $\varepsilon>0$. This is the so called Gauss
circle problem.

H.Cramer in \cite{Cramer} removed the $\varepsilon$ in the theorem of Hardy
and proved the more precise asymptotic estimate
\[
\lim_{T\rightarrow+\infty}\left\{  T^{-3/2}%
{\displaystyle\int_{0}^{T}}
\left\vert
{\displaystyle\sum_{n\leq t}}
r(n)-\pi t\right\vert ^{2}dt\right\}  ^{1/2}=\left\{  \dfrac{1}{3\pi^{2}}%
{\displaystyle\sum_{n=1}^{+\infty}}
\frac{r(n)^{2}}{n^{3/2}}\right\}  ^{1/2}.
\]

The distribution and higher power moment in the Gauss circle problem and the
related Dirichlet divisor problem have been studied by D.R.Heath-Brown in
\cite{Heath} and by K.M.Tsang in \cite{tsang}.
See also \cite{Jarnik} for a tridimensional version of these results.

The above results for the disc and the ball have been extended to other domains. In \cite{Nowak85}, W.Nowak  proved that if $\Omega$ is a convex set in the plane
with smooth boundary with strictly positive curvature, then for every $R\geq
1$,
\[
\left\{  \dfrac{1}{R}{%
{\displaystyle\int_{0}^{R}}
}\left\vert \mathcal{D}\left(  r\Omega\right)  \right\vert ^{2}dr\right\}
^{1/2}\leq CR^{1/2}.
\]

Indeed, P.Bleher proved in \cite{B} a more precise asymptotic estimate: There
exists an explicit constant $C$ such that
\[
\lim_{R\rightarrow+\infty}R^{-1/2}\left\{  \dfrac{1}{R}{%
{\displaystyle\int_{0}^{R}}
}\left\vert \mathcal{D}\left(  r\Omega\right)  \right\vert ^{2}dr\right\}
^{1/2}=C.
\]

M.Huxley in \cite{Huxley1} considered the mean value of the discrepancy over
short intervals and proved that if $\Omega$ is a convex set in the plane with
smooth boundary with strictly positive curvature, then
\[
\left\{  {%
{\displaystyle\int_{R}^{R+1}}
}\left\vert \mathcal{D}\left(  r\Omega\right)  \right\vert ^{2}dr\right\}
^{1/2}\leq CR^{1/2}\log^{1/2}\left(  R\right)  .
\]

W.Nowak in \cite{Nowak} proved that the above estimate remains valid also when
the integration is over the interval $R\leq r\leq R+\log\left(  R\right)  $,
while for $H\leq R$ but $H/\log\left(  R\right)  \rightarrow+\infty$ he proved
the more precise asymptotic estimate
\[
\lim_{R\rightarrow+\infty}R^{-1/2}\left\{  \dfrac{1}{H}{%
{\displaystyle\int_{R}^{R+H}}
}\left\vert \mathcal{D}\left(  r\Omega\right)  \right\vert ^{2}dr\right\}
^{1/2}=C.
\]

A.Iosevich, E.Sawyer, A.Seeger in \cite{ISS} and \cite{ISS2} extended the above results to
convex sets in $\mathbb{R}^{d}$ with smooth boundary with strictly positive
Gaussian curvature,
\[
\left\{  \dfrac{1}{R}{%
{\displaystyle\int_{R}^{2R}}
}\left\vert \mathcal{D}\left(  r\Omega\right)  \right\vert ^{2}dr\right\}
^{1/2}\leq\left\{
\begin{array}
[c]{ll}%
CR^{1/2} & \text{if }d=2\text{,}\\
CR\log\left(  R\right)  & \text{if }d=3\text{,}\\
CR^{d-2} & \text{if }d>3\text{.}%
\end{array}
\right.
\]

The above are results on the averages of the discrepancy under dilations.
D.G.Kendall considered the mean square average of the discrepancy under
translations and proved in \cite{Kendall} that if $\Omega$ is an oval in
$\mathbb{R}^{d}$,
\[
\left\{  {%
{\displaystyle\int_{\mathbb{T}^{d}}}
}\left\vert \mathcal{D}\left(  R\Omega-x\right)  \right\vert ^{2}dx\right\}
^{1/2}\leq CR^{\left(  d-1\right)  /2}.
\]

Mean square averages of the discrepancy under rotations of the domain have been considered
by A.Iosevich in \cite{Ios}, and more general random discrepancies have been discussed in \cite{IK}.

The study of the $L^{p}$ norm of the discrepancy with $p\neq2$ is more recent
and the results are less complete.
In \cite{BCT} the authors studied the $L^{p}$ norm of the discrepancy for
rotated and translated polyhedra $\Omega$ in $\mathbb{R}^{d}$,
\[
\left\{  {%
{\displaystyle\int_{\mathbb{SO}\left(  d\right)  }}
}{%
{\displaystyle\int_{\mathbb{T}^{d}}}
}\left\vert \mathcal{D}\left(  \sigma R\Omega-x\right)  \right\vert
^{p}dxd\sigma\right\}  ^{1/p}\leq%
\begin{cases}
C\log^{d}\left(  R\right)  & \text{if }p=1\text{,}\\
CR^{\left(  d-1\right)  \left(  1-1/p\right)  } & \text{if }1<p\leq
+\infty\text{.}%
\end{cases}
\]

In the same paper it was also proved that the above inequalities can be
reversed, at least for a simplex and for $p>1$. In other words, the $L^{p}$
discrepancy of a polyhedron grows with $p$. On the contrary, for certain
domains with curvature there exists a range of indices $p$ where the $L^{p}$
discrepancy is of the same order as the $L^{2}$ discrepancy, possibly up to a
logarithmic transgression. Indeed, M.Huxley in \cite{Huxley3} proved that if
$\Omega$ is a convex set in the plane with boundary with continuous positive
curvature, then
\[
\left\{  {%
{\displaystyle\int_{\mathbb{T}^{2}}}
}\left\vert \mathcal{D}\left(  R\Omega-x\right)  \right\vert ^{4}dx\right\}
^{1/4}\leq CR^{1/2}\log^{1/4}\left(  R\right)  .
\]

In \cite{BCGT} the authors extended the above result to convex sets with
smooth boundary with positive Gaussian curvature in higher dimensions,
\[
\left\{  {%
{\displaystyle\int_{\mathbb{T}^{d}}}
}\left\vert \mathcal{D}\left(  R\Omega-x\right)  \right\vert ^{p}dx\right\}
^{1/p}\leq%
\begin{cases}
CR^{\left(  d-1\right)  /2} & \text{if }p<2d/\left(  d-1\right)  \text{,}\\
CR^{\left(  d-1\right)  /2}\log^{(d-1)/2d}\left(  R\right)  & \text{if
}p=2d/\left(  d-1\right)  \text{.}%
\end{cases}
\]

The present paper continues this line of research.
Let us conclude this section with a few examples.

\begin{example}
If $d\mu\left(  r\right)  $ is the unit mass concentrated at $r=0$, then
$\widehat{\mu}\left(  \zeta\right)  =1$, so that $\beta=0$, and the
$L^{p}\left(  L^{2}\right)  $ mixed norm in Theorem \ref{1} reduces to a pure
$L^{p}$ norm, and one obtains
\begin{align*}
&  \left\{  \int_{\mathbb{T}^{d}}\left\vert R^{-\left(  d-1\right)
/2}\mathcal{D}\left(  R\Omega-x\right)  \right\vert ^{p}dx\right\}  ^{1/p}\\
&  \leq\left\{
\begin{array}
[c]{ll}%
C\left(  2d/\left(  d-1\right)  -p\right)  ^{-\left(  d-1\right)  /2d} &
\text{if }d\geq2\text{ and }p<2d/\left(  d-1\right)  \text{,}\\
C\log^{\left(  d-1\right)  /2d}\left(  1+R\right)  & \text{if }d\geq2\text{
and }p=2d/\left(  d-1\right)  \text{.}%
\end{array}
\right.
\end{align*}
In particular, one recovers some of the results in \cite{Huxley3} and
\cite{BCGT}.
\end{example}

\begin{example}
If $d\mu\left(  r\right)  $ is the uniformly distributed measure in $\left\{
0<r<1\right\}  $, then the $L^{p}\left(  L^{2}\right)  $ mixed norm in
Theorem \ref{1} is
\[
\left\{  \int_{\mathbb{T}^{d}}\left( \dfrac{1}{H}%
{\displaystyle\int_{R}^{R+H}}
\left\vert r^{-\left(  d-1\right)  /2}\mathcal{D}\left(  r\Omega-x\right)
\right\vert ^{2}dr\right) ^{p/2}dx\right\}  ^{1/p}.
\]
The Fourier transform of the uniformly distributed measure in $\left\{
0<r<1\right\}  $ has decay $\beta=1$,
\[
\widehat{\mu}\left(  \zeta\right)  =%
{\displaystyle\int_{0}^{1}}
\exp\left(  -2\pi i\zeta r\right)  dr=\exp\left(  -\pi i\zeta\right)
\dfrac{\sin\left(  \pi\zeta\right)  }{\pi\zeta}.
\]
On the other hand, if $\psi\left(  r\right)  $\ is a non negative smooth
function with integral one and support in $0\leq r\leq1$, one can consider a
smoothed average
\[
\left\{  \int_{\mathbb{T}^{d}}\left( \int_{\mathbb{R}}\left\vert r^{-\left(
d-1\right)  /2}\mathcal{D}\left(  r\Omega-x\right)  \right\vert ^{2}H^{-1}%
\psi\left(  H^{-1}\left(  r-R\right)  \right)  dr\right)  ^{p/2}dx\right\}
^{1/p}.
\]
This smoothed average is equivalent to the uniform average over $\left\{
R<r<R+H\right\}  $, but the decay of the Fourier transform $\widehat{\psi
}\left(  \zeta\right)  $ is faster than any power $\beta$. Hence for the
uniformly distributed measure in $\left\{  0<r<1\right\}  $ the theorem
applies with the indexes corresponding to $\beta>1$:
\begin{align*}
&  \left\{  \int_{\mathbb{T}^{d}}\left( \dfrac{1}{H}%
{\displaystyle\int_{R}^{R+H}}
\left\vert r^{-\left(  d-1\right)  /2}\mathcal{D}\left(  r\Omega-x\right)
\right\vert ^{2}dr\right)  ^{p/2}dx\right\}  ^{1/p}\\
&  \leq\left\{
\begin{array}
[c]{ll}%
Cp^{1/2} & \text{if }d=2\text{ and }p<+\infty\text{,}\\
C\log^{1/2}\left(  1+R\right)  & \text{if }d=2\text{ and }p=+\infty\text{,}\\
C\left(  6-p\right)  ^{-1/3} & \text{if }d=3\text{ and }p<6\text{,}\\
C\log^{1/3}\left(  1+R\right)  & \text{if }d=3\text{ and }p=6\text{,}\\
C\left(  \left(  2d-4\right)  /\left(  d-3\right)  -p\right)  ^{-\left(
d-3\right)  /\left(  2d-4\right)  } & \text{if }d\geq4\text{ and }p<\left(
2d-4\right)  /\left(  d-3\right)  \text{,}\\
C\log^{\left(  d-3\right)  /\left(  2d-4\right)  }\left(  1+R\right)  &
\text{if }d\geq4\text{ and }p=\left(  2d-4\right)  /\left(  d-3\right)
\text{.}%
\end{array}
\right.
\end{align*}
Observe that the range of indexes in the above theorem and corollaries for
which the mixed $L^{p}\left(  L^{2}\right)  $ norm remains uniformly bounded
is larger than the range of indexes in \cite{Huxley3} and \cite{BCGT} quoted in the previous example.
\end{example}

\begin{example}
As an intermediate case between the two preceeding examples, one can consider a measure $d\mu\left(  r\right)=r^{-\alpha}\chi_{\left\{  0<r<1\right\}
}\left(  r\right)  dr$, with $0<\alpha<1$. In this case $\left\vert \widehat{\mu
}\left(  \zeta\right)  \right\vert \leq C\left(  1+\left\vert \zeta\right\vert
\right)  ^{\alpha-1}$, that is $\beta=1-\alpha$.
As a more sophisticated intermediate example, recall that a probability measure is a Salem measure if its Fourier dimension $\gamma
=\sup\left\{  \delta:\ \left\vert \widehat{\mu}\left(  \zeta\right)
\right\vert \leq C\left(  1+\left\vert \zeta\right\vert \right)  ^{-\delta
/2}\ \right\}  $ is equal to the Hausdorff dimension of the support. Such
measures exist for every dimension $0<\gamma<1$. See \cite[Section 12.17]{PM}. 
The above theorem and
corollary assert that the discrepancy cannot be too large in mean on the
supports of translated and dilates of these measures.
\end{example}

\section{Proofs of theorems and corollaries}

The proofs will be splitted into a number of lemmas, some of them well known.
The starting point is the observation of D.G.Kendall that the discrepancy
$\mathcal{D}\left(  r\Omega-x\right)  $ is a periodic function of the
translation, and it has a Fourier expansion with coefficients that are a
sampling of the Fourier transform of $\Omega$,
\[
\widehat{\chi}_{\Omega}\left(  \xi\right)  =%
{\displaystyle\int_{\Omega}}
\exp\left(  -2\pi i\xi x\right)  dx.
\]

\begin{lemma}
\label{Fourier} The number of integer points in $r\Omega-x$, a translated by a
vector $x\in\mathbb{R}^{d}$\ and dilated by a factor $r>0$\ of a domain
$\Omega$\ in the $d$ dimensional Euclidean space, is a periodic function of the
translation with Fourier expansion
\[%
{\displaystyle\sum_{k\in\mathbb{Z}^{d}}}
\chi_{r\Omega-x}(k)=%
{\displaystyle\sum_{n\in\mathbb{Z}^{d}}}
r^{d}\widehat{\chi}_{\Omega}\left(  rn\right)  \exp(2\pi inx).
\]
In particular,
\[
\mathcal{D}\left(  r\Omega-x\right)  =%
{\displaystyle\sum_{n\in\mathbb{Z}^{d}\setminus\left\{  0\right\}  }}
r^{d}\widehat{\chi}_{\Omega}\left(  rn\right)  \exp(2\pi inx).
\]

\end{lemma}

\begin{proof}
This is a particular case of the Poisson summation formula.
\end{proof}

\begin{remark}\label{r1}
We emphasize that the Fourier expansion of the discrepancy converges at least
in $L^{2}\left(  \mathbb{T}^{d}\right)  $, but we are not claiming that it
converges pointwise. Indeed, the discrepancy is discontinuous, hence the
associated Fourier expansion does not converge absolutely or uniformly. To
overcome this problem, one can introduce a mollified discrepancy. If the
domain $\Omega$\ is convex and contains the origin, then there exists
$\varepsilon>0$\ such that if $\varphi\left(  x\right)  $\ is a non negative
smooth radial function with support in $\left\{  \left\vert x\right\vert
\leq\varepsilon\right\}  $\ and with integral 1, and if $0<\delta\leq1$\ and
$r\geq1$, then
\begin{gather*}
\left\vert \Omega\right\vert \left(  \left(  r-\delta\right)  ^{d}%
-r^{d}\right)  +\left(  r-\delta\right)  ^{d}%
{\displaystyle\sum_{n\in\mathbb{Z}^{d}\setminus\left\{  0\right\}  }}
\widehat{\varphi}\left(  \delta n\right)  \widehat{\chi}_{\Omega}\left(
\left(  r-\delta\right)  n\right)  \exp\left(  2\pi inx\right) \\
\leq%
{\displaystyle\sum_{n\in\mathbb{Z}^{d}}}
\chi_{r\Omega}(n+x)-\left\vert \Omega\right\vert r^{d}\\
\leq\left\vert \Omega\right\vert \left(  \left(  r+\delta\right)  ^{d}%
-r^{d}\right)  +\left(  r+\delta\right)  ^{d}%
{\displaystyle\sum_{n\in\mathbb{Z}^{d}\setminus\left\{  0\right\}  }}
\widehat{\varphi}\left(  \delta n\right)  \widehat{\chi}_{\Omega}\left(
\left(  r+\delta\right)  n\right)  \exp\left(  2\pi inx\right)  .
\end{gather*}
One has $\left\vert \left(  r+\delta\right)  ^{d}-r^{d}\right\vert \leq
Cr^{d-1}\delta$, and one can define the mollified discrepancy
\[
\left(  r\pm\delta\right)  ^{d}%
{\displaystyle\sum_{n\in\mathbb{Z}^{d}\setminus\left\{  0\right\}  }}
\widehat{\varphi}\left(  \delta n\right)  \widehat{\chi}_{\Omega}\left(
\left(  r\pm\delta\right)  n\right)  \exp\left(  2\pi inx\right)  .
\]
Observe that the discrepancy is the limit of this mollified discrepancy as
$\delta\rightarrow0+$. Also observe that since $\left\vert \widehat{\varphi
}\left(  \zeta\right)  \right\vert \leq C\left(  1+\left\vert \zeta\right\vert
\right)  ^{-\gamma}$ for every $\gamma>0$, the mollified Fourier
expansion has no problems of convergence.
\end{remark}

\begin{lemma}
\label{Fourier Asymptotic copy(1)} Assume that $\Omega$\ is a convex body in
$\mathbb{R}^{d}$\ with smooth boundary and everywhere positive Gaussian
curvature. Define the support function $g\left(  x\right)  =\sup_{y\in\Omega
}\left\{  x\cdot y\right\}  $. Then, there exist functions $\left\{
a_{j}\left(  \xi\right)  \right\}  _{j=0}^{+\infty}$ and $\left\{
b_{j}\left(  \xi\right)  \right\}  _{j=0}^{+\infty}$ homogeneous of degree $0$
and smooth in $\mathbb{R}^{d}\setminus\left\{  0\right\}  $ such that the Fourier
transform of the characteristic function of $\Omega$ for $\left\vert
\xi\right\vert \rightarrow+\infty$ has the asymptotic expansion
\begin{align*}
&  \widehat{\chi}_{\Omega}\left(  \xi\right)  =%
{\displaystyle\int_{\Omega}}
\exp\left(  -2\pi i\xi\cdot x\right)  dx\\
&  =\exp\left(  -2\pi ig\left(  \xi\right)  \right)  \left\vert \xi\right\vert
^{-\left(  d+1\right)  /2}%
{\displaystyle\sum_{j=0}^{h}}
a_{j}\left(  \xi\right)  \left\vert \xi\right\vert ^{-j}+\exp\left(  2\pi
ig\left(  -\xi\right)  \right)  \left\vert \xi\right\vert ^{-\left(
d+1\right)  /2}%
{\displaystyle\sum_{j=0}^{h}}
b_{j}\left(  \xi\right)  \left\vert \xi\right\vert ^{-j}\\
&  +\mathcal{O}\left(  \left\vert \xi\right\vert ^{-\left(  d+2h+3\right)
/2}\right)  .
\end{align*}
The functions $a_{j}\left(  \xi\right)  $ and $b_{j}\left(  \xi\right)  $
depend on a finite number of derivatives of a parametrization of the boundary
of $\Omega$ at the points with outward unit normal $\pm\xi/\left\vert
\xi\right\vert $. In particular, $a_{0}\left(  \xi\right)  $ and $b_{0}\left(
\xi\right)  $ are, up to some absolute constants, equal to $K\left(  \pm
\xi\right)  ^{-1/2}$, with $K\left(  \pm\xi\right)  $\ the Gaussian curvature
of $\partial\Omega$\ at the points with outward unit normal $\pm\xi/\left\vert
\xi\right\vert $.
\end{lemma}

\begin{proof}
This is a classical result. See e.g. \cite{GGV}, 
\cite{Herz_decad}, \cite{Hlawka}, \cite{Stein}. Here, as an explicit example, we just recall that the
Fourier transform of a ball $\left\{  x\in\mathbb{R}^{d}:\left\vert
x\right\vert \leq R\right\}  $\ can be expressed in terms of a Bessel
function, and Bessel functions have simple asymptotic expansions in terms of
trigonometric functions,
\begin{align*}
&  \widehat{\chi}_{\left\{  \left\vert x\right\vert \leq R\right\}  }\left(
\xi\right)  =R^{d}\widehat{\chi}_{\left\{  \left\vert x\right\vert
\leq1\right\}  }\left(  R\xi\right)  =R^{d}\left\vert R\xi\right\vert
^{-d/2}J_{d/2}\left(  2\pi\left\vert R\xi\right\vert \right) \\
&  \approx\pi^{-1}R^{(d-1)/2}\left\vert \xi\right\vert ^{-\left(  d+1\right)
/2}\cos\left(  2\pi R\left\vert \xi\right\vert -\left(  d+1\right)
\pi/4\right) \\
&  -2^{-4}\pi^{-2}\left(  d^{2}-1\right)  R^{(d-3)/2}\left\vert \xi\right\vert
^{-\left(  d+3\right)  /2}\sin\left(  2\pi R\left\vert \xi\right\vert -\left(
d+1\right)  \pi/4\right)  +...
\end{align*}
More generally, also the Fourier transform of an ellipsoid, that is an affine
image of a ball, can be expressed in terms of Bessel functions. See \cite{SW}.
\end{proof}

\begin{lemma}
\label{Asymptotic Discrepancy} Assume that $\Omega$\ is a convex body in
$\mathbb{R}^{d}$\ with smooth boundary and everywhere positive Gaussian
curvature. Let $z$ be a complex parameter, and for every $j=0,1,2,...$ and
$r\geq1$, with the notation of the previous lemmas, let define the tempered
distributions $\Phi_{j}\left(  z,r,x\right)  $ via the Fourier expansions
\begin{align*}
&  \Phi_{j}\left(  z,r,x\right)  =r^{-j}%
{\displaystyle\sum_{n\in\mathbb{Z}^{d}\setminus\left\{  0\right\}  }}
a_{j}\left(  n\right)  \left\vert n\right\vert ^{-z-j}\exp\left(  -2\pi
ig\left(  n\right)  r\right)  \exp\left(  2\pi inx\right) \\
&  +r^{-j}%
{\displaystyle\sum_{n\in\mathbb{Z}^{d}\setminus\left\{  0\right\}  }}
b_{j}\left(  n\right)  \left\vert n\right\vert ^{-z-j}\exp\left(  2\pi
ig\left(  -n\right)  r\right)  \exp\left(  2\pi inx\right)  .
\end{align*}

(1) If $\operatorname{Re}\left(  z\right)  +j>d/2$ then the Fourier expansion
that defines $\Phi_{j}\left(  z,r,x\right)  $ converges in $L^{2}\left(
\mathbb{T}^{d}\right)  $. If $\operatorname{Re}\left(  z\right)  +j>d$ then
the convergence is absolute and uniform.

(2) Let
\[
\mathcal{R}_{h}\left(  r,x\right)  =r^{-\left(  d-1\right)  /2}\mathcal{D}%
\left(  r\Omega-x\right)  -%
{\displaystyle\sum_{j=0}^{h}}
\Phi_{j}\left(  \left(  d+1\right)  /2,r,x\right)  .
\]
If $h>\left(  d-3\right)  /2$ there exists $C$ such that for every $r\geq1$,
\[
\left\vert \mathcal{R}_{h}\left(  r,x\right)  \right\vert \leq Cr^{-h-1}.
\]

\end{lemma}

\begin{proof}
This is a consequence of the previous lemmas. The terms $\Phi
_{j}\left(  \left(  d+1\right)  /2,r,x\right)  $ come from the terms
homogeneous of degree $-\left(  d+1\right)  /2-j$ in the asymptotic expansion
of the Fourier transform of $\Omega$, while the remainder $\mathcal{R}%
_{h}\left(  r,x\right)  $ is given by an absolutely and uniformly convergent
Fourier expansion.
\end{proof}

 After these preliminary lemmas and with the above notation, we have that the
 normalized discrepancy has the asymptotic expansion
\[
r^{-\left(  d-1\right)  /2}\mathcal{D}%
\left(  r\Omega-x\right)=%
{\displaystyle\sum_{j=0}^{h}}
\Phi_{j}\left(  \left(  d+1\right)  /2,r,x\right)+O(r^{-h-1})  .
\]
 
Now one can describe the strategy of the proof of Theorem \ref{1} as follows: The terms that appear in the expansion of the
normalized discrepancy have the form
\[%
r^{-j}{\displaystyle\sum_{n\in\mathbb{Z}^{d}\setminus\left\{  0\right\}  }}
c(n)\left\vert n\right\vert ^{-\left(  d+1\right)  /2-j}\exp\left(  \pm2\pi
ig\left(  \mp n\right)  r\right)  \exp\left(  2\pi inx\right)  .
\]

One can replace the real parameter $\left(  d+1\right)  /2$ which describes
the decay of the Fourier transform by a complex parameter $z$, and define the
function
\[
\Theta_j\left(  z,r,x\right)  =r^{-j}%
{\displaystyle\sum_{n\in\mathbb{Z}^{d}\setminus\left\{  0\right\}  }}
c(n)\left\vert n\right\vert ^{-z-j}\exp\left(  \pm2\pi ig\left(  \mp n\right)
r\right)  \exp\left(  2\pi inx\right)  .
\]
Observe that the Fourier coefficients of this function are analytic functions
of the complex variable $z$. One can estimate the $L^{2}\left(  L^{2}\right)
$ norm of this function via the Parseval equality, and it can be easily seen that the norm is finite if
$\operatorname{Re}\left(  z\right)  >d/2$. Then one can estimate the
$L^{p}\left(  L^{2}\right)  $ norm with $p\geq4$ via the Hausdorff Young
inequality, and it turns out that  if $\operatorname{Re}\left(  z\right)
>d\left(  1-1/p\right)  -1/2$ then this norm is finite. Finally, the result for $z=\left(  d+1\right)
/2$, the one for the discrepancy, follows by complex interpolation.

\begin{lemma}
\label{Support} Let $g\left(  x\right)  =\sup_{y\in\Omega}\left\{  x\cdot
y\right\}  $ be the support function of a convex $\Omega$ which contains the
origin, and with a smooth boundary and everywhere positive Gaussian curvature.

(1) This support function is convex, homogeneous of degree one, positive and
smooth away from the origin, and it is equivalent to the Euclidean norm, that
is there exist $0<A<B$ such that for every $x$,
\[
A\left\vert x\right\vert \leq g\left(  x\right)  \leq B\left\vert x\right\vert
.
\]

(2) There exists $C>0$ such that for all unit vectors $\omega$ and $\vartheta$
in $\mathbb{R}^{d}$, there exists a number $A\left(  \vartheta,\omega\right)
$ such that for every real $\tau$ one has
\[
\left\vert g\left(  \vartheta-\tau\omega\right)  -g\left(  \vartheta\right)
\right\vert \geq C\frac{\left\vert \tau\right\vert \left\vert \tau-A\left(
\vartheta,\omega\right)  \right\vert }{1+\left\vert \tau\right\vert }.
\]
\end{lemma}

\begin{proof}
When $\Omega$ is the sphere $\{|x|\leq 1\}$ the proof is explicit and transparent. 
Indeed, one has $g(x)=|x|$ and, if $|\vartheta|=|\omega|=1$, then 
\[
|\vartheta-\tau\omega|-|\vartheta|=\frac{|\vartheta-\tau\omega|^2-|\vartheta|^2}{|\vartheta-\tau\omega|+|\vartheta|}=\frac{\tau(\tau-2\vartheta\cdot\omega)}{1+|\vartheta-\tau\omega|}.
\] 
The proof for a generic convex set is a bit more involved.
The convexity of the support function  easily follows from the convexity of
$\Omega$, and also the other properties in (1) are elementary. In order to prove (2),
observe that for $\left\vert \omega\right\vert =\left\vert \vartheta
\right\vert =1$ and $-\infty<\tau<+\infty$,
\begin{equation*}
 \left\vert g\left(  \vartheta-\tau\omega\right)  -g\left(  \vartheta
\right)  \right\vert =\left\vert \frac{g\left(  \vartheta-\tau\omega\right)
^{2}-g\left(  \vartheta\right)  ^{2}}{g\left(  \vartheta-\tau\omega\right)
+g\left(  \vartheta\right)  }\right\vert 
 \geq C\frac{\left\vert g\left(  \vartheta-\tau\omega\right)  ^{2}-g\left(
\vartheta\right)  ^{2}\right\vert }{1+\left\vert \tau\right\vert }.
\end{equation*}
It then suffices to prove that there exists a $C>0$ such that for all unit
vectors $\vartheta$ and $\omega$, there exists a number $A\left(
\vartheta,\omega\right)  $ such that for every real $\tau$ one has
\[
\left\vert g\left(  \vartheta-\tau\omega\right)  ^{2}-g\left(  \vartheta
\right)  ^{2}\right\vert \geq C\left\vert \tau\right\vert \left\vert
\tau-A\left(  \vartheta,\omega\right)  \right\vert .
\]
Let us show that the function $f\left(  \tau\right)  =g\left(  \vartheta
-\tau\omega\right)  ^{2}-g\left(  \vartheta\right)  ^{2}$ is strictly convex.
If $\omega=\pm\vartheta$, then
\[
f\left(  \tau\right)  =g\left(  \left(  1\pm\tau\right)  \vartheta\right)
^{2}-g\left(  \vartheta\right)  ^{2}=\left(  \left(  1\pm\tau\right)
^{2}-1\right)  g\left(  \vartheta\right)  ^{2}=\left(  \tau^{2}\pm
2\tau\right)  g\left(  \vartheta\right)  ^{2}.
\]
Therefore, if $\omega=\pm\vartheta$,
\[
\dfrac{d^{2}}{d\tau^{2}}f\left(  \tau\right)  =2g\left(  \vartheta\right)
^{2}\geq2C>0.
\]
If $\omega\neq\pm\vartheta$, then $\vartheta-\tau\omega\neq0$, and
\[
\dfrac{d}{d\tau}f\left(  \tau\right)  =-2g\left(  \vartheta-\tau\omega\right)
\nabla g\left(  \vartheta-\tau\omega\right)  \cdot\omega,
\]%
\[
\dfrac{d^{2}}{d\tau^{2}}f\left(  \tau\right)  =2\left(  \nabla g\left(
\vartheta-\tau\omega\right)  \cdot\omega\right)  ^{2}+2g\left(  \vartheta
-\tau\omega\right)  \omega^{t}\cdot\nabla^{2}g\left(  \vartheta-\tau
\omega\right)  \cdot\omega.
\]
For notational simplicity call $\vartheta-\tau\omega=x$. The Hessian matrix
$\nabla^{2}g\left(  x\right)  $ is homogeneous of degree $-1$ and positive
semidefinite. When $\left\vert x\right\vert =1$ one eigenvalue is $0$ and the
associated eigenvector is the gradient $\nabla g\left(  x\right)  $, while all
the other eigenvalues are the reciprocal of the principal curvatures at the
point where the normal is $\nabla g\left(  x\right)  $. See \cite[Corollary
2.5.2]{Schneider}. Let $\alpha$ be the minimum of $g\left(  x\right)  $ on the
sphere $\left\{  \left\vert x\right\vert =1\right\}  $, let $\beta>0$ be the
minimum of $\left\vert \nabla g\left(  x\right)  \right\vert $ on $\left\{
\left\vert x\right\vert =1\right\}  $, and let $\gamma>0$ be the minimum of
the non zero eigenvalues of $\nabla^{2}g\left(  x\right)  $ on $\left\{
\left\vert x\right\vert =1\right\}  $. If one decomposes $\omega$ into
$\omega_{0}+\omega_{1}$, where $\omega_{0}$ is parallel to $\nabla g\left(
x\right)  $ and $\omega_{1}$ is orthogonal to $\nabla g\left(  x\right)  $,
then
\begin{align*}
&  \dfrac{d^{2}}{d\tau^{2}}f\left(  \tau\right)  =2\left(  \nabla g\left(
x\right)  \cdot\omega_{0}\right)  ^{2}+2g\left(  x\right)  \omega_{1}^{t}%
\cdot\nabla^{2}g\left(  x\right)  \cdot\omega_{1}\\
&  =2\left\vert \nabla g\left(  x\right)  \right\vert ^{2}\left\vert
\omega_{0}\right\vert ^{2}+2g\left(  x/\left\vert x\right\vert \right)
\omega_{1}^{t}\cdot\nabla^{2}g\left(  x/\left\vert x\right\vert \right)
\cdot\omega_{1}\\
&  \geq2\beta^{2}\left\vert \omega_{0}\right\vert ^{2}+2\alpha\gamma\left\vert
\omega_{1}\right\vert ^{2}\geq2C>0.
\end{align*}
Therefore, for every $\vartheta$ and $\omega$ the function $f\left(
\tau\right)  $ is strictly convex, and it has exactly two zeros, one is
$\tau=0$ and the other is $\tau=A\left(  \vartheta,\omega\right)  $, possibly
the zero is double. By the Lagrange remainder in interpolation, there exists
$\varepsilon$ such that
\[
f\left(  \tau\right)  =\tau\left(  \tau-A\left(  \vartheta,\omega\right)
\right)  \dfrac{1}{2}\dfrac{d^{2}f}{d\tau^{2}}\left(  \varepsilon\right)  .
\]
 And since $d^{2}f\left(  \tau\right)  /d\tau^{2}\geq2C>0$,
\[
\left\vert g\left(  \vartheta-\tau\omega\right)  ^{2}-g\left(  \vartheta
\right)  ^{2}\right\vert \geq C\left\vert \tau\right\vert
\left\vert \tau-A\left(  \vartheta,\omega\right)  \right\vert .
\]
\end{proof}

\begin{lemma}
\label{Integral} If $g\left(  x\right)  $ is the support function of $\Omega$,
then for every $\left(  d+1\right)  /2\leq\alpha<d$ and $\beta\geq0$, there
exists $C$\ such that for every $y\in\mathbb{R}^{d}-\left\{  0\right\}  $,
\begin{align*}
&
{\displaystyle\int_{\mathbb{R}^{d}}}
\left\vert x\right\vert ^{-\alpha}\left\vert x-y\right\vert ^{-\alpha}\left(
1+\left\vert g\left(  x\right)  -g\left(  x-y\right)  \right\vert \right)
^{-\beta}dx\\
&  \leq\left\{
\begin{array}
[c]{ll}%
C\left\vert y\right\vert ^{d-2\alpha-\beta} & \text{if }0\leq\beta<1\text{,}\\
C\left\vert y\right\vert ^{d-2\alpha-1}\log\left(  2+\left\vert y\right\vert
\right)  & \text{if }\beta=1\text{,}\\
C\left\vert y\right\vert ^{d-2\alpha-1} & \text{if }\beta>1\text{.}%
\end{array}
\right.
\end{align*}

\end{lemma}

\begin{proof}
Let us explain the numerology behind the lemma. If there is no
cutoff $\left(  1+\left\vert g\left(  x\right)  -g\left(  x-y\right)
\right\vert \right)  ^{-\beta}$, then the change of variables $x=\left\vert
y\right\vert z$ and $y=\left\vert y\right\vert \omega$ gives
\[%
{\displaystyle\int_{\mathbb{R}^{d}}}
\left\vert x\right\vert ^{-\alpha}\left\vert x-y\right\vert ^{-\alpha
}dx=\left\vert y\right\vert ^{d-2\alpha}%
{\displaystyle\int_{\mathbb{R}^{d}}}
\left\vert z\right\vert ^{-\alpha}\left\vert z-\omega\right\vert ^{-\alpha
}dx=C\left\vert y\right\vert ^{d-2\alpha}.
\]
On the other hand, the cutoff $\left(  1+\left\vert g\left(  x\right)
-g\left(  x-y\right)  \right\vert \right)  ^{-\beta}$ gives an extra
decay. In particular, the integral with the cutoff $\left(  1+\left\vert
g\left(  x\right)  -g\left(  x-y\right)  \right\vert \right)  ^{-\beta}$ with
$\beta$ large is essentially over the set $\left\{  g\left(  x\right)
=g\left(  x-y\right)  \right\}  $, that is the cutoff reduces the space
dimension by $1$. This suggests that, at least when $\beta$ is large, the
integral with the cutoff can be seen as the convolution in $\mathbb{R}^{d-1}$
of two homogeneous functions of degree $-\alpha$, and this gives the decay
$\left\vert y\right\vert ^{d-1-2\alpha}$. Hence, when $\beta=0$ the decay is
$\left\vert y\right\vert ^{d-2\alpha}$, and when $\beta>1$ the decay is
$\left\vert y\right\vert ^{d-1-2\alpha}$. Finally, by interpolation, when $0<\beta<1$
the decay is $\left\vert y\right\vert ^{d-\beta-2\alpha}$. This is just the
numerology, the details of the proof are more delicate. The change of
variables $x=\left\vert y\right\vert z$ and $y=\left\vert y\right\vert \omega$
gives
\begin{align*}
&
{\displaystyle\int_{\mathbb{R}^{d}}}
\left\vert x\right\vert ^{-\alpha}\left\vert x-y\right\vert ^{-\alpha}\left(
1+\left\vert g\left(  x\right)  -g\left(  x-y\right)  \right\vert \right)
^{-\beta}dx\\
&  =\left\vert y\right\vert ^{d-2\alpha}%
{\displaystyle\int_{\mathbb{R}^{d}}}
\left\vert z\right\vert ^{-\alpha}\left\vert z-\omega\right\vert ^{-\alpha
}\left(  1+\left\vert y\right\vert \left\vert g\left(  z\right)  -g\left(
z-\omega\right)  \right\vert \right)  ^{-\beta}dz.
\end{align*}
If $\varepsilon$ is positive and suitably small, there exists $\delta>0$ such
that for every $\omega$ and $z$ with $\left\vert \omega\right\vert =1$ and
$\left\vert z\right\vert <\varepsilon$ one has $g\left(  z-\omega\right)
-g\left(  z\right)  >\delta$. This suggests to split the domain of integration into
\[
\left\{  \left\vert z\right\vert \leq\varepsilon\right\}  \cup\left\{
\left\vert z-\omega\right\vert \leq\varepsilon\right\}  \cup\left\{
\left\vert z\right\vert \geq\varepsilon,\ \left\vert z-\omega\right\vert
\geq\varepsilon\right\}  .
\]
The integral over  $\left\{  \left\vert z\right\vert \leq
\varepsilon\right\}  $ is bounded by
\begin{align*}
&
{\displaystyle\int_{\left\{  \left\vert z\right\vert \leq\varepsilon\right\}
}}
\left\vert z\right\vert ^{-\alpha}\left\vert z-\omega\right\vert ^{-\alpha
}\left(  1+\left\vert y\right\vert \left\vert g\left(  z\right)  -g\left(
z-\omega\right)  \right\vert \right)  ^{-\beta}dz\\
&  \leq\left(  1-\varepsilon\right)  ^{-\alpha}\left(  1+\delta\left\vert
y\right\vert \right)  ^{-\beta}%
{\displaystyle\int_{\left\{  \left\vert z\right\vert \leq\varepsilon\right\}
}}
\left\vert z\right\vert ^{-\alpha}dz\\
&  \leq C\left(  1+\left\vert y\right\vert \right)  ^{-\beta}.
\end{align*}
The integral over $\left\{  \left\vert z-\omega\right\vert
\leq\varepsilon\right\}  $ is bounded similarly,
\[%
{\displaystyle\int_{\left\{  \left\vert z-\omega\right\vert \leq
\varepsilon\right\}  }}
\left\vert z\right\vert ^{-\alpha}\left\vert z-\omega\right\vert ^{-\alpha
}\left(  1+\left\vert y\right\vert \left\vert g\left(  z\right)  -g\left(
z-\omega\right)  \right\vert \right)  ^{-\beta}dz\leq C\left(  1+\left\vert
y\right\vert \right)  ^{-\beta}.
\]
It remains to estimate the integral over $\left\{  \left\vert z\right\vert
\geq\varepsilon,\ \left\vert z-\omega\right\vert \geq\varepsilon\right\}  $.
First observe that
\begin{align*}
&
{\displaystyle\int_{\left\{  \left\vert z\right\vert \geq\varepsilon
,\ \left\vert z-\omega\right\vert \geq\varepsilon\right\}  }}
\left\vert z\right\vert ^{-\alpha}\left\vert z-\omega\right\vert ^{-\alpha
}\left(  1+\left\vert y\right\vert \left\vert g\left(  z\right)  -g\left(
z-\omega\right)  \right\vert \right)  ^{-\beta}dz\\
&  \leq C%
{\displaystyle\int_{\left\{  \left\vert z\right\vert \geq\varepsilon\right\}
}}
\left\vert z\right\vert ^{-2\alpha}\left(  1+\left\vert y\right\vert
\left\vert g\left(  z\right)  -g\left(  z-\omega\right)  \right\vert \right)
^{-\beta}dz.
\end{align*}
In spherical coordinates write $z=\rho\vartheta$, with $\varepsilon\leq
\rho<+\infty$ and $\left\vert \vartheta\right\vert =1$, and $dz=\rho
^{d-1}d\rho d\vartheta$, with $d\vartheta$ the surface measure on the $d-1$
dimensional sphere $\mathbb{S}^{d-1}=\left\{  \left\vert \vartheta\right\vert
=1\right\}  $. Then by the above lemma and the change of variables
$\rho=1/\tau$, recalling that $\left\vert \omega\right\vert =1$ and
$2\alpha-d-1\geq0$,
\begin{align*}
&
{\displaystyle\int_{\left\{  \left\vert z\right\vert \geq\varepsilon\right\}
}}
\left\vert z\right\vert ^{-2\alpha}\left(  1+\left\vert y\right\vert
\left\vert g\left(  z\right)  -g\left(  z-\omega\right)  \right\vert \right)
^{-\beta}dz\\
&  =%
{\displaystyle\int_{\mathbb{S}^{d-1}}}
{\displaystyle\int_{\varepsilon}^{+\infty}}
\rho^{d-1-2\alpha}\left(  1+\left\vert y\right\vert \left\vert g\left(
\rho\vartheta\right)  -g\left(  \rho\vartheta-\omega\right)  \right\vert
\right)  ^{-\beta}d\rho d\vartheta\\
&  =%
{\displaystyle\int_{\mathbb{S}^{d-1}}}
{\displaystyle\int_{0}^{1/\varepsilon}}
\tau^{2\alpha-d-1}\left(  1+\left\vert y\right\vert \left\vert \dfrac{g\left(
\vartheta\right)  -g\left(  \vartheta-\tau\omega\right)  }{\tau}\right\vert
\right)  ^{-\beta}d\tau d\vartheta\\
&  \leq \varepsilon^{d+1-2\alpha}%
{\displaystyle\int_{\mathbb{S}^{d-1}}}
{\displaystyle\int_{0}^{1/\varepsilon}}
\left(  1+\left\vert y\right\vert \left\vert \dfrac{g\left(  \vartheta\right)
-g\left(  \vartheta-\tau\omega\right)  }{\tau}\right\vert \right)  ^{-\beta
}d\tau d\vartheta\\
&  \leq C%
{\displaystyle\int_{\mathbb{S}^{d-1}}}
{\displaystyle\int_{0}^{1/\varepsilon}}
\left(  1+\left\vert y\right\vert \left\vert \tau-A\left(  \vartheta
,\omega\right)  \right\vert \right)  ^{-\beta}d\tau d\vartheta.
\end{align*}
Finally, if $\sup_{\left\vert \vartheta\right\vert =\left\vert \omega
\right\vert =1}\left\{  \left\vert A\left(  \vartheta,\omega\right)
\right\vert \right\}  =\gamma$, then
\begin{align*}
&
{\displaystyle\int_{\mathbb{S}^{d-1}}}
{\displaystyle\int_{0}^{1/\varepsilon}}
\left(  1+\left\vert y\right\vert \left\vert \tau-A\left(  \vartheta
,\omega\right)  \right\vert \right)  ^{-\beta}d\tau d\vartheta\\
&  \leq\left\vert \mathbb{S}^{d-1}\right\vert
{\displaystyle\int_{-\gamma-1/\varepsilon}^{\gamma+1/\varepsilon}}
\left(  1+\left\vert y\right\vert \left\vert \tau\right\vert \right)
^{-\beta}d\tau\\
&  =\left\vert \mathbb{S}^{d-1}\right\vert \left\vert y\right\vert ^{-1}%
{\displaystyle\int_{-\left(  \gamma+1/\varepsilon\right)  \left\vert
y\right\vert }^{\left(  \gamma+1/\varepsilon\right)  \left\vert y\right\vert
}}
\left(  1+\left\vert \tau\right\vert \right)  ^{-\beta}d\tau\\
&  \leq\left\{
\begin{array}
[c]{ll}%
C\left(  1+\left\vert y\right\vert \right)  ^{-\beta} & \text{if }0\leq
\beta<1\text{,}\\
C\left(  1+\left\vert y\right\vert \right)  ^{-1}\log\left(  2+\left\vert
y\right\vert \right)  & \text{if }\beta=1\text{,}\\
C\left(  1+\left\vert y\right\vert \right)  ^{-1} & \text{if }\beta>1\text{.}%
\end{array}
\right.
\end{align*}

\end{proof}

By Lemma \ref{Asymptotic Discrepancy}, the normalized discrepancy $r^{-\left(  d-1\right)
/2}\mathcal{D}\left(  r\Omega-x\right)  $ is a sum of terms $\Phi_{j}\left(
\left(  d+1\right)  /2,r,x\right)  $. The following lemma studies the two
terms $\Theta_{j}\left(  z,r,x\right)$ that appear in the definition of $\Phi_{j}\left(  z,r,x\right)  $.

\begin{lemma}
\label{Terms} Let $d\mu\left(  r\right)  $\ be a Borel probability measure on
$\mathbb{R}$ with support in $0<\varepsilon<r<\delta<+\infty$, and with
$\left\vert \widehat{\mu}\left(  \zeta\right)  \right\vert \leq B\left(
1+\left\vert \zeta\right\vert \right)  ^{-\beta}$. Let $c\left(  \xi\right)  $
be a bounded homogeneous function of degree $0$, let $ \left(  d+1\right)  /2\leq\operatorname{Re}\left(
z\right) <d$, and for $j=0,1,2,...$ and any choice of
$\pm$, define
\[
\Theta_{j}\left(  z,r,x\right)  =r^{-j}\sum_{n\in\mathbb{Z}^{d}\setminus\left\{
0\right\}  }c\left(  n\right)  \left\vert n\right\vert ^{-z-j}\exp\left(
\pm2\pi ig\left(  \mp n\right)  r\right)  \exp\left(  2\pi inx\right)  .
\]
Moreover, for $H,R\geq1$, define%
\[
\mathcal{F}_{j}\left(  z,H,R,x\right)  =\int_{\mathbb{R}}\left\vert \Theta
_{j}\left(  z,r,x\right)  \right\vert ^{2}d\mu_{H,R}\left(  r\right)  .
\]
Expand this last function into a Fourier series in the variable $x$,
\[
\mathcal{F}_{j}\left(  z,H,R,x\right)  =\sum_{k\in\mathbb{Z}^{d}}%
\widehat{\mathcal{F}}_{j}\left(  z,H,R,k\right)  \exp\left(  2\pi ikx\right)
.
\]
(1) If $j=0$ there exists a constant $C$, which may depend on $d$, $B$,
$\beta$ and $\sup_{n\in \mathbb{Z}^{d}}\{|c(n)|\}$, but it is independent of the complex parameter $z$ and of the real parameters
$H$ and $R$, such that the
Fourier coefficients of $\mathcal{F}_{0}\left(  z,H,R,x\right)  $ satisfy for
every $H,R\geq1$ and $k\in\mathbb{Z}^{d}$ the estimates
\[
\left\vert \widehat{\mathcal{F}}_{0}\left(  z,H,R,k\right)  \right\vert
\leq\left\{
\begin{array}
[c]{ll}%
C\left(  1+\left\vert k\right\vert \right)  ^{d-\beta-2\operatorname{Re}%
\left(  z\right)  } & \text{if }0\leq\beta<1\text{,}\\
C\left(  1+\left\vert k\right\vert \right)  ^{d-1-2\operatorname{Re}\left(
z\right)  }\log\left(  2+\left\vert k\right\vert \right)  & \text{if }%
\beta=1\text{,}\\
C\left(  1+\left\vert k\right\vert \right)  ^{d-1-2\operatorname{Re}\left(
z\right)  } & \text{if }\beta>1\text{.}%
\end{array}
\right.
\]
(2) If $j\geq1$ then there exists a constant $C$ such that the Fourier
coefficients of $\mathcal{F}_{j}\left(  z,H,R,x\right)  $ satisfy for every
$H,R\geq1$ and $k\in\mathbb{Z}^{d}$ the estimates
\[
\left\vert \widehat{\mathcal{F}}_{j}\left(  z,H,R,k\right)  \right\vert \leq
C\left(  R+H\right)  ^{-2j}\left(  1+\left\vert k\right\vert \right)
^{d-1-2\operatorname{Re}\left(  z\right)  }.
\]

\end{lemma}

\begin{proof}
Let us fix a choice of $\pm$,
\[
\Theta_{j}\left(  z,r,x\right)  =r^{-j}\sum_{n\in\mathbb{Z}^{d}\setminus \left\{
0\right\}  }c\left(  n\right)  \left\vert n\right\vert ^{-z-j}\exp\left(
-2\pi ig\left(  n\right)  r\right)  \exp\left(  2\pi inx\right)  .
\]
Expanding the product $\Theta_{j}\left(  z,r,x\right)  \cdot\overline
{\Theta_{j}\left(  z,r,x\right)  }$ and integrating against $d\mu_{H,R}\left(
r\right)  $, one obtains
\begin{align*}
&  \mathcal{F}_{j}\left(  z,H,R,x\right)  =%
{\displaystyle\int_{\mathbb{R}}}
\left\vert \Theta_{j}\left(  z,r,x\right)  \right\vert ^{2}d\mu_{H,R}\left(
r\right) \\
&  =\sum_{k\in\mathbb{Z}^{d}}\sum_{n\in\mathbb{Z}^{d}\setminus\left\{  0,k\right\}
}c\left(  n\right)  \overline{c\left(  n-k\right)}  \left\vert n\right\vert
^{-z-j}\left\vert n-k\right\vert ^{-\overline{z}-j}\exp\left(  2\pi ikx\right)
\\
&  \times\exp\left(  2\pi i\left(  g\left(  n-k\right)  -g\left(  n\right)
\right)  R\right)
{\displaystyle\int_{\mathbb{R}}}
\left(  R+Hr\right)  ^{-2j}\exp\left(  2\pi iH\left(  g\left(  n-k\right)
-g\left(  n\right)  \right)  r\right)  d\mu\left(  r\right)  .
\end{align*}
The product term by term of the series and the integration term by term can be
justified with a suitable summation method, which amounts to introduce a
cutoff in the the series that defines $\Theta_{j}\left(  z,r,x\right)  $. See
 Remark \ref{r1}. In particular, the Fourier coefficients
of $\mathcal{F}_{j}\left(  z,H,R,x\right)  $ are
\begin{align*}
&  \widehat{\mathcal{F}}_{j}\left(  z,H,R,k\right)  =\sum_{n\in\mathbb{Z}%
^{d}\setminus\left\{  0,k\right\}  }c\left(  n\right)  \overline{c\left(  n-k\right)}  \left\vert
n\right\vert ^{-z-j}\left\vert n-k\right\vert ^{-\overline{z}-j}\\
&  \times\exp\left(  2\pi i\left(  g\left(  n-k\right)  -g\left(  n\right)
\right)  R\right)
{\displaystyle\int_{\mathbb{R}}}
\left(  R+Hr\right)  ^{-2j}\exp\left(  2\pi iH\left(  g\left(  n-k\right)
-g\left(  n\right)  \right)  r\right)  d\mu\left(  r\right)  .
\end{align*}
When $j=0$, by the assumption on the Fourier
transform of the measure $d\mu\left(  r\right)  $,
\begin{align*}
&  \left\vert \widehat{\mathcal{F}}_{0}\left(  z,H,R,k\right)  \right\vert \\
&  \leq B\sup_{n\in\mathbb{Z}^{d}}\left\{  \left\vert c\left(  n\right)
\right\vert ^{2}\right\}  \sum_{n\in\mathbb{Z}^{d}\setminus\left\{  0,k\right\}
}\left\vert n\right\vert ^{-\operatorname{Re}\left(  z\right)  }\left\vert
n-k\right\vert ^{-\operatorname{Re}\left(  z\right)  }\left(  1+\left\vert
g\left(  n-k\right)  -g\left(  n\right)  \right\vert \right)  ^{-\beta}\\
&  \leq C\int_{\mathbb{R}^{d}}\left\vert x\right\vert ^{-\operatorname{Re}%
\left(  z\right)  }\left\vert x-k\right\vert ^{-\operatorname{Re}\left(
z\right)  }\left(  1+\left\vert g\left(  x-k\right)  -g\left(  x\right)
\right\vert \right)  ^{-\beta}dx.
\end{align*}
It then suffices to apply Lemma \ref{Integral}. 
The substitution of the series with an integral can be justified observing that
all functions involved are slowly varying when one replaces $n$ with an $x$
such that $|n-x|\leq C$.

When $j\geq1$,
\begin{align*}
&  \left\vert \widehat{\mathcal{F}}_{j}\left(  z,H,R,k\right)  \right\vert \\
&  \leq\sup_{n\in\mathbb{Z}^{d}}\left\{  \left\vert c\left(  n\right)
\right\vert ^{2}\right\}  \sum_{n\in\mathbb{Z}^{d}\setminus\left\{  0,k\right\}
}\left\vert n\right\vert ^{-\operatorname{Re}\left(  z\right)  -j}\left\vert
n-k\right\vert ^{-\operatorname{Re}\left(  z\right)  -j}%
{\displaystyle\int_{\mathbb{R}}}
\left(  R+Hr\right)  ^{-2j}d\mu\left(  r\right)\\
& \leq\begin{cases}
C\left(  R+H\right)  ^{-2j}\left(  1+\left\vert k\right\vert \right)
^{d-2j-2\operatorname{Re}\left(  z\right)  } & \text{ if }\operatorname{Re}\left(  z\right)+j<d, \\
C\left(  R+H\right)  ^{-2j}\left(  1+\left\vert k\right\vert \right)
^{-d }\log(2+|k|) & \text{ if }\operatorname{Re}\left(  z\right)+j=d, \\
C\left(  R+H\right)  ^{-2j}\left(  1+\left\vert k\right\vert \right)
^{-\operatorname{Re}\left(  z\right) -j }& \text{ if }\operatorname{Re}\left(  z\right)+j>d.
\end{cases}
\end{align*}
Finally observe that when $\alpha<d$ and $j\geq 1$, then all these estimates are dominated by 
$C\left(  R+H\right)  ^{-2j}\left(  1+\left\vert k\right\vert \right)
^{d-1-2\operatorname{Re}\left(  z\right)  }.$

\end{proof}

\begin{lemma}
\label{L(p)-5} Let $0\leq\beta<1$ and let $\mathcal{F}_{0}\left(
z,H,R,x\right)  $ be as in Lemma \ref{Terms}. Then there exists $C$ such that
for every $H,R\geq1$ the following hold.
\begin{itemize}
\item[(1)] If $d/2<\operatorname{Re}\left(  z\right)  \leq\left(  3d-2\beta\right)
/4$ and $2\leq p<\left(  d-2\beta\right)  /\left(  d-\beta-\operatorname{Re}%
\left(  z\right)  \right)  $, then
\[
\left\{  \int_{\mathbb{T}^{d}}\left\vert \mathcal{F}_{0}\left(
z,H,R,x\right)  \right\vert ^{p/2}dx\right\}  ^{1/p}\leq C\left(
\dfrac{d-2\beta}{d-\beta-\operatorname{Re}\left(  z\right)  }-p\right)
^{-\left(  d-\beta-\operatorname{Re}\left(  z\right)  \right)  /\left(
d-2\beta\right)  }.
\]
\item[(2)] If $\left(  3d-2\beta\right)  /4<\operatorname{Re}\left(  z\right)
<d-\beta/2$ and $4\leq p<d/\left(  d-\beta/2-\operatorname{Re}\left(
z\right)  \right)  $, then%
\[
\left\{  \int_{\mathbb{T}^{d}}\left\vert \mathcal{F}_{0}\left(
z,H,R,x\right)  \right\vert ^{p/2}dx\right\}  ^{1/p}\leq C\left(  \dfrac
{d}{d-\beta/2-\operatorname{Re}\left(  z\right)  }-p\right)  ^{\left(
d-\beta/2-\operatorname{Re}\left(  z\right)  \right)  /d-1/2}.
\]
\item[(3)] If $\operatorname{Re}\left(  z\right)  =d-\beta/2$ and $p<+\infty$, then
\[
\left\{  \int_{\mathbb{T}^{d}}\left\vert \mathcal{F}_{0}\left(
z,H,R,x\right)  \right\vert ^{p/2}dx\right\}  ^{1/p}\leq Cp^{1/2}.
\]
\end{itemize}
\end{lemma}

\begin{proof}
In this and in the following proofs, we shall repeatedly use an elementary inequality.
If $\delta>d$, then
\[
\sum_{k\in\mathbb Z^d}(1+|k|)^{-\delta}\leq C(\delta-d)^{-1}.
\]
If $p=2$ and $\operatorname{Re}\left(  z\right)  >d/2$ then, by Parseval
equality,
\begin{align*}
&  \left\{  \int_{\mathbb{T}^{d}}\left\vert \mathcal{F}_{0}\left(
z,H,R,x\right)  \right\vert ^{p/2}dx\right\}  ^{1/p}\\
&  =\left\{  \int_{\mathbb{T}^{d}}\int_{\mathbb{R}}\left\vert \Theta
_{0}\left(  z,r,x\right)  \right\vert ^{2}d\mu_{H,R}\left(  r\right)
dx\right\}  ^{1/p}\\
&  =\left\{  \int_{\mathbb{R}}d\mu\left(  r\right)  \sum_{n\in\mathbb{Z}%
^{d}\setminus\left\{  0\right\}  }\left\vert c\left(  n\right)  \right\vert
^{2}\left\vert n\right\vert ^{-2\operatorname{Re}\left(  z\right)  }\right\}
^{1/2}\\
&  \leq C\left(  \operatorname{Re}\left(  z\right)  -\dfrac{d}{2}\right)
^{-1/2}.
\end{align*}
If $4\leq p\leq+\infty$ and $\operatorname{Re}\left(  z\right)  >d\left(
1-1/p\right)  -\beta/2$ then $p/2>2$ and, by the Hausdorff Young inequality,
\begin{align*}
&  \left\{  \int_{\mathbb{T}^{d}}\left\vert \mathcal{F}_{0}\left(
z,H,R,x\right)  \right\vert ^{p/2}dx\right\}  ^{1/p}\\
&  \leq\left\{
{\displaystyle\sum_{k\in\mathbb{Z}^{d}}}
\left\vert \widehat{\mathcal{F}}_{0}\left(  z,H,R,k\right)  \right\vert
^{p/(p-2)}\right\}  ^{\left(  p-2\right)  /2p}\\
&  \leq C\left\{
{\displaystyle\sum_{k\in\mathbb{Z}^{d}}}
\left\vert \left(  1+\left\vert k\right\vert \right)  ^{d-\beta
-2\operatorname{Re}\left(  z\right)  }\right\vert ^{p/(p-2)}\right\}
^{\left(  p-2\right)  /2p}\\
&  \leq C\left(  \left(  2\operatorname{Re}\left(  z\right)  +\beta-d\right)
\dfrac{p}{p-2}-d\right)  ^{-\left(  p-2\right)  /2p}\\
&  =C\left(  \dfrac{p-2}{2p}\right)  ^{\left(  p-2\right)  /2p}\left(
\operatorname{Re}\left(  z\right)  +\dfrac{\beta}{2}-d\left(  1-\dfrac{1}%
{p}\right)  \right)  ^{-\left(  p-2\right)  /2p}\\
&  \leq C\left(  \operatorname{Re}\left(  z\right)  +\dfrac{\beta}{2}-d\left(
1-\dfrac{1}{p}\right)  \right)  ^{1/p-1/2}.
\end{align*}
This proves the cases $p=2$ and $p\geq4$. The case $2<p<4$ follows from these
cases via complex interpolation of vector valued $L\left(  p\right)  $ spaces.
For the definition of the complex interpolation method, see for example
\cite[Chapter 4 and Chapter 5]{BL}. Here we recall the
relevant result: Let $\mathbb{H}$ be a Hilbert space and $\mathbb{X}$ a
measure space, let $1\leq a<b\leq+\infty$, $-\infty<A<B<+\infty$, and let
$\Theta\left(  z\right)  $ be a function with values in the vector valued
space $L^{a}\left(  \mathbb{X},\mathbb{H}\right)  +L^{b}\left(  \mathbb{X}%
,\mathbb{H}\right)  $, continuous and bounded on the closed strip $\left\{
A\leq\operatorname{Re}\left(  z\right)  \leq B\right\}  $ and analytic on the
open strip $\left\{  A<\operatorname{Re}\left(  z\right)  <B\right\}  $.
Assume that there exist constants $M$ and $N$ such that for every
$-\infty<t<+\infty$,
\[
\left\{
\begin{array}
[c]{c}%
\left\Vert \Theta\left(  A+it\right)  \right\Vert _{L^{a}\left(
\mathbb{X},\mathbb{H}\right)  }\leq M,\\
\left\Vert \Theta\left(  B+it\right)  \right\Vert _{L^{b}\left(
\mathbb{X},\mathbb{H}\right)  }\leq N.
\end{array}
\right.
\]
If $1/p=\left(  1-\vartheta\right)  /a+\vartheta/b$, with $0<\vartheta<1$,
then
\[
\left\Vert \Theta\left(  \left(  1-\vartheta\right)  A+\vartheta B\right)
\right\Vert _{L^{p}\left(  \mathbb{X},\mathbb{H}\right)  }\leq M^{1-\vartheta
}N^{\vartheta}.
\]
Here the analytic function is $\Theta_{0}\left(  z,r,x\right)  $, the Hilbert
space $\mathbb H$ is $L^{2}\left(  \mathbb{R},d\mu_{H,R}\left(  r\right)  \right)  $, the
measure space $\mathbb X$ is the torus $\mathbb{T}^{d}$, $a=2$, $A=d/2+\varepsilon$,
$b=4$, $B=3d/4-\beta/2+\varepsilon$, with $\varepsilon>0$, $M=C\varepsilon
^{-1/2}$ and $N=C\varepsilon^{-1/4}$. By the above computations,
\begin{align*}
&  \left\{  \int_{\mathbb{T}^{d}}\left\vert \mathcal{F}_{0}\left(
z,H,R,x\right)  \right\vert ^{p/2}dx\right\}  ^{1/p}\\
&  \leq\left\{
\begin{array}
[c]{ll}%
C\varepsilon^{-1/2} & \text{if }p=2\text{ and }\operatorname{Re}\left(
z\right)  =d/2+\varepsilon\text{,}\\
C\varepsilon^{-1/4} & \text{if }p=4\text{ and }\operatorname{Re}\left(
z\right)  =3d/4-\beta/2+\varepsilon\text{.}%
\end{array}
\right.
\end{align*}
By complex interpolation with
$$\begin{cases}
1/p=\left(  1-\vartheta\right)  /2+\vartheta/4,\\
\operatorname{Re}\left(  z\right)  =\left(  1-\vartheta\right)  \left(
d/2+\varepsilon\right)  +\vartheta\left(  3d/4-\beta/2+\varepsilon\right)  ,
\end{cases}$$
that is, $2<p<4$ and $\operatorname{Re}\left(  z\right)  =d\left(
1-1/p\right)  +2\beta/p-\beta+\varepsilon$, one obtains
\begin{equation*}
  \left\{  \int_{\mathbb{T}^{d}}\left\vert \mathcal{F}_{0}\left(
z,H,R,x\right)  \right\vert ^{p/2}dx\right\}  ^{1/p}
  \leq C\left(  \operatorname{Re}\left(  z\right)  +\beta-\dfrac{2\beta}%
{p}-d\left(  1-\dfrac{1}{p}\right)  \right)  ^{-1/p}.
\end{equation*}
The estimates that we have obtained blow up when $\operatorname{Re}\left(
z\right)  \rightarrow \mathrm{critical}\left(  z\right)  +$,
\[
\operatorname{Re}\left(  z\right)  \rightarrow\left\{
\begin{array}
[c]{ll}%
d\left(  1-\dfrac{1}{p}\right)  +\dfrac{2\beta}{p}-\beta+ & \text{if }%
p\leq4\text{,}\\
d\left(  1-\dfrac{1}{p}\right)  -\dfrac{\beta}{2}+ & \text{if }p\geq4.\\
\end{array}
\right.
\]
This is the same as $p\rightarrow \mathrm{critical}\left(  p\right)  -$,
\[
p\rightarrow\left\{
\begin{array}
[c]{ll}%
\dfrac{d-2\beta}{d-\beta-\operatorname{Re}\left(  z\right)  }- & \text{if
} d/2<\operatorname{Re}(z)\leq(3d-2\beta)/4\text{,}\\
\dfrac{d}{d-\beta/2-\operatorname{Re}\left(  z\right)  }- & \text{if }%
(3d-2\beta)/4\leq\operatorname{Re}(z)\leq d-\beta/2.\\
\end{array}
\right.
\]
In order to complete the proof of the lemma, it suffices to translate the estimates 
 for $\operatorname{Re}(z)\rightarrow \mathrm{critical}\left(  z\right)  +$
into estimates for $p\rightarrow \mathrm{critical}\left(  p\right)  -$.

If $d/2<\operatorname{Re}%
\left(  z\right)  \leq\left(  3d-2\beta\right)  /4$ and $2\leq p<\left(
d-2\beta\right)  /\left(  d-\beta-\operatorname{Re}\left(  z\right)  \right)
$, then $2\leq p\leq4$ and $\operatorname{Re}\left(  z\right)  >d\left(
1-1/p\right)  +2\beta/p-\beta$, and one has
\begin{align*}
&  \left(  \operatorname{Re}\left(  z\right)  +\beta-\dfrac{2\beta}%
{p}-d\left(  1-\dfrac{1}{p}\right)  \right)  ^{-1/p}\\
&  =p^{1/p}\left(  d-\beta-\operatorname{Re}\left(  z\right)  \right)
^{-1/p}\left(  \dfrac{d-2\beta}{d-\beta-\operatorname{Re}\left(  z\right)
}-p\right)  ^{\left(  d-\beta-\operatorname{Re}\left(  z\right)  \right)
/\left(  d-2\beta\right)  -1/p}\\
&  \times\left(  \dfrac{d-2\beta}{d-\beta-\operatorname{Re}\left(  z\right)
}-p\right)  ^{-\left(  d-\beta-\operatorname{Re}\left(  z\right)  \right)
/\left(  d-2\beta\right)  }\\
&  \leq C\left(  \dfrac{d-2\beta}{d-\beta-\operatorname{Re}\left(  z\right)
}-p\right)  ^{-\left(  d-\beta-\operatorname{Re}\left(  z\right)  \right)
/\left(  d-2\beta\right)  }.
\end{align*}
We have used the inequalities $x^{1/x}\leq e^{1/e}$ for every $x>0$, and
$\left(  x-y\right)  ^{1/x-1/y}=\left(  \left(  x-y\right)  ^{-\left(
x-y\right)  }\right)  ^{1/xy}\leq\left(  e^{1/e}\right)  ^{1/xy}\leq e^{1/e}$
for every $x>y\geq1$. Observe that the above constant $C$ may depend on $d$,
$\beta$, $\operatorname{Re}\left(  z\right)  $, but it is independent of $p$.

If $\left(  3d-2\beta\right)  /4<\operatorname{Re}\left(  z\right)
<d-\beta/2$ and $4\leq p<d/\left(  d-\beta/2-\operatorname{Re}\left(
z\right)  \right)  $, then $\operatorname{Re}\left(  z\right)  >d\left(
1-1/p\right)  -\beta/2$,  and one has
\begin{align*}
&  \left(  \operatorname{Re}\left(  z\right)  +\dfrac{\beta}{2}-d\left(
1-\dfrac{1}{p}\right)  \right)  ^{1/p-1/2}\\
&  =p^{1/2-1/p}\left(  d-\left(  d-\beta/2-\operatorname{Re}\left(  z\right)
\right)  p\right)  ^{1/p-1/2}\\
&  =p^{1/2-1/p}\left(  d-\beta/2-\operatorname{Re}\left(  z\right)  \right)
^{1/p-1/2}\left(  \dfrac{d}{d-\beta/2-\operatorname{Re}\left(  z\right)
}-p\right)  ^{1/p-\left(  d-\beta/2-\operatorname{Re}\left(  z\right)
\right)  /d}\\
&  \times\left(  \dfrac{d}{d-\beta/2-\operatorname{Re}\left(  z\right)
}-p\right)  ^{\left(  d-\beta/2-\operatorname{Re}\left(  z\right)  \right)
/d-1/2}\\
&  \leq C\left(  \dfrac{d}{d-\beta/2-\operatorname{Re}\left(  z\right)
}-p\right)  ^{\left(  d-\beta/2-\operatorname{Re}\left(  z\right)  \right)
/d-1/2}.
\end{align*}
If $\operatorname{Re}\left(  z\right)  =d-\beta/2$ and $p<+\infty$, then
one has
\[
\left(  \operatorname{Re}\left(  z\right)  +\dfrac{\beta}{2}-d\left(
1-\dfrac{1}{p}\right)  \right)  ^{1/p-1/2}=d^{1/p-1/2}p^{-1/p}p^{1/2}\leq
Cp^{1/2}.
\]
\end{proof}

\begin{lemma}
\label{L(p)-6} Let $\beta=1$ and let $\mathcal{F}_{0}\left(  z,H,R,x\right)  $
be as in Lemma \ref{Terms}. Then there exists $C$ such that for every
$H,R\geq1$ the following hold.
\begin{itemize}
\item[(1)] If $d/2<\operatorname{Re}\left(  z\right)  \leq\left(  3d-2\right)  /4$
and $2\leq p<(d-2)/(d-1-\operatorname{Re}\left(  z\right)  )$, then
\[
\left\{  \int_{\mathbb{T}^{d}}\left\vert \mathcal{F}_{0}\left(
z,H,R,x\right)  \right\vert ^{p/2}dx\right\}  ^{1/p}\leq C\left(  \dfrac
{d-2}{d-1-\operatorname{Re}\left(  z\right)  }-p\right)  ^{-\left(
\operatorname{Re}\left(  z\right)  -1\right)  /\left(  d-2\right)  }.
\]
\item[(2)] If $\left(  3d-2\right)  /4<\operatorname{Re}\left(  z\right)  <d-1/2$ and
$4\leq p<d/\left(  d-\operatorname{Re}\left(  z\right)  -1/2\right)  $, then%
\[
\left\{  \int_{\mathbb{T}^{d}}\left\vert \mathcal{F}_{0}\left(
z,H,R,x\right)  \right\vert ^{p/2}dx\right\}  ^{1/p}\leq C\left(  \frac
{d}{d-1/2-\operatorname{Re}\left(  z\right)  }-p\right)  ^{-\left(
2\operatorname{Re}\left(  z\right)  +1\right)  /2d}.
\]
\item[(3)] If $\operatorname{Re}\left(  z\right)  =d-1/2$ and $p<+\infty$, then
\[
\left\{  \int_{\mathbb{T}^{d}}\left\vert \mathcal{F}_{0}\left(
z,H,R,x\right)  \right\vert ^{p/2}dx\right\}  ^{1/p}\leq Cp.
\]
\end{itemize}
\end{lemma}
\begin{proof}
The proof is, \textit{mutatis mutandis}, as in the previous lemma, only observe 
the different exponents in the right hand side of the estimates. lf $p=2$
and $\operatorname{Re}\left(  z\right)  >d/2$, then
\[
\left\{  \int_{\mathbb{T}^{d}}\left\vert \mathcal{F}_{0}\left(
z,H,R,x\right)  \right\vert ^{p/2}dx\right\}  ^{1/p}\leq C\left(
\operatorname{Re}\left(  z\right)  -\dfrac{d}{2}\right)  ^{-1/2}.
\]
If $4\leq p\leq+\infty$ and $\operatorname{Re}\left(  z\right)  >d\left(
1-1/p\right)  -1/2$ then, by the Hausdorff Young inequality,
\begin{align*}
&  \left\{  \int_{\mathbb{T}^{d}}\left\vert \mathcal{F}_{0}\left(
z,H,R,x\right)  \right\vert ^{p/2}dx\right\}  ^{1/p}\\
&  \leq\left\{
{\displaystyle\sum_{k\in\mathbb{Z}^{d}}}
\left\vert \widehat{\mathcal{F}}_{0}\left(  z,H,R,k\right)  \right\vert
^{p/(p-2)}\right\}  ^{\left(  p-2\right)  /2p}\\
&  \leq C\left\{
{\displaystyle\sum_{k\in\mathbb{Z}^{d}}}
\left\vert \left(  1+\left\vert k\right\vert \right)  ^{d-1-2\operatorname{Re}%
\left(  z\right)  }\log\left(  2+\left\vert k\right\vert \right)  \right\vert
^{p/(p-2)}\right\}  ^{\left(  p-2\right)  /2p}.
\end{align*}
The series can be compared with the integral
\begin{align*}
&  \left\{  \int_{\mathbb{R}^{d}}\left\vert \left(  1+\left\vert x\right\vert
\right)  ^{d-1-2\operatorname{Re}\left(  z\right)  }\log\left(  2+\left\vert
x\right\vert \right)  \right\vert ^{p/(p-2)}dx\right\}  ^{\left(  p-2\right)
/2p}\\
&  =\left\{  \left\vert \left\{  \left\vert \vartheta\right\vert =1\right\}
\right\vert
{\displaystyle\int_{0}^{+\infty}}
\left(  1+\rho\right)  ^{\left(  d-1-2\operatorname{Re}\left(  z\right)
\right)  p/(p-2)}\log^{p/(p-2)}\left(  2+\rho\right)  \rho^{d-1}d\rho\right\}
^{\left(  p-2\right)  /2p}.
\end{align*}
The last integral can be compared to another integral,
\[%
{\displaystyle\int_{1}^{+\infty}}
t^{-\alpha}\log^{\beta}\left(  t\right)  dt=\left(  \alpha-1\right)
^{-\left(  \beta+1\right)  }%
{\displaystyle\int_{0}^{+\infty}}
s^{\beta}e^{-s}ds=\left(  \alpha-1\right)  ^{-\left(  \beta+1\right)  }%
\Gamma\left(  \beta+1\right)  .
\]
Hence,
\begin{align*}
&  \left\{
{\displaystyle\int_{0}^{+\infty}}
\left(  1+\rho\right)  ^{\left(  d-1-2\operatorname{Re}\left(  z\right)
\right)  p/(p-2)+d-1}\log^{p/(p-2)}\left(  2+\rho\right)  d\rho\right\}
^{\left(  p-2\right)  /2p}\\
&  \leq C\left(  \left(  2\operatorname{Re}\left(  z\right)  +1-d\right)
\dfrac{p}{p-2}-d\right)  ^{-\left(  1+p/(p-2)\right)  \left(  p-2\right)
/2p}\\
&  \leq C\left(  \operatorname{Re}\left(  z\right)  -d\left(  1-\dfrac{1}%
{p}\right)  +\dfrac{1}{2}\right)  ^{1/p-1}.
\end{align*}
This proves the cases $p=2$ and $p\geq4$. The case $2<p<4$ follows from these
cases via complex interpolation. By the above computations,
\begin{align*}
&  \left\{  \int_{\mathbb{T}^{d}}\left\vert \mathcal{F}_{0}\left(
z,H,R,x\right)  \right\vert ^{p/2}dx\right\}  ^{1/p}\\
&  \leq\left\{
\begin{array}
[c]{ll}%
C\varepsilon^{-1/2} & \text{if }p=2\text{ and }\operatorname{Re}\left(
z\right)  =d/2+\varepsilon\text{,}\\
C\varepsilon^{-3/4} & \text{if }p=4\text{ and }\operatorname{Re}\left(
z\right)  =\left(  3d-2\right)  /4+\varepsilon\text{.}%
\end{array}
\right.
\end{align*}
By complex interpolation, if $2<p<4$ and $\operatorname{Re}\left(  z\right)
=d\left(  1-1/p\right)  +2/p-1+\varepsilon$,
\begin{equation*}
 \left\{  \int_{\mathbb{T}^{d}}\left\vert \mathcal{F}_{0}\left(
z,H,R,x\right)  \right\vert ^{p/2}dx\right\}  ^{1/p}
  \leq C\left(  \operatorname{Re}\left(  z\right)  -\left(  d\left(
1-\dfrac{1}{p}\right)  +\dfrac{2}{p}-1\right)  \right)  ^{1/p-1}.
\end{equation*}
The estimates that we have obtained blow up when $\operatorname{Re}\left(
z\right)  \rightarrow \mathrm{critical}\left(  z\right)  +$,
\[
\operatorname{Re}\left(  z\right)  \rightarrow\left\{
\begin{array}
[c]{ll}%
d\left(  1-\dfrac{1}{p}\right)  +\dfrac{2}{p}-1+ & \text{if }%
p\leq4\text{,}\\
d\left(  1-\dfrac{1}{p}\right)  -\dfrac{1}{2}+ & \text{if }p\geq4.\\
\end{array}
\right.
\]
This is the same as $p\rightarrow \mathrm{critical}\left(  p\right)  -$,
\[
p\rightarrow\left\{
\begin{array}
[c]{ll}%
\dfrac{d-2}{d-1-\operatorname{Re}\left(  z\right)  }- & \text{if
} d/2<\operatorname{Re}(z)\leq(3d-2)/4\text{,}\\
\dfrac{d}{d-1/2-\operatorname{Re}\left(  z\right)  }- & \text{if }%
(3d-2)/4\leq\operatorname{Re}(z)\leq d-1/2.\\
\end{array}
\right.
\]
In order to complete the proof it suffices to translate, as in the previous lemma, the estimates 
 for $\operatorname{Re}(z)\rightarrow \mathrm{critical}\left(  z\right)  +$
into estimates for $p\rightarrow \mathrm{critical}\left(  p\right)  -$.
\end{proof}

\begin{lemma}
\label{L(p)-7} Let $\beta>1$ and let $\mathcal{F}_{0}\left(  z,H,R,x\right)  $
be as in the Lemma \ref{Terms}. Then there exists $C$ such that for every
$H,R\geq1$ the following hold.
\begin{itemize}
\item[(1)] If $d/2<\operatorname{Re}\left(  z\right)  \leq\left(  3d-2\right)  /4$
and $2\leq p<(d-2)(d-1-\operatorname{Re}\left(  z\right)  )$, then
\[
\left\{  \int_{\mathbb{T}^{d}}\left\vert \mathcal{F}_{0}\left(
z,H,R,x\right)  \right\vert ^{p/2}dx\right\}  ^{1/p}\leq C\left(  \dfrac
{d-2}{d-1-\operatorname{Re}\left(  z\right)  }-p\right)  ^{-\left(
d-1-\operatorname{Re}\left(  z\right)  \right)  /\left(  d-2\right)  }.
\]
\item[(2)] If $\left(  3d-2\right)  /4<\operatorname{Re}\left(  z\right)  <d-1/2$ and
$4\leq p<d/\left(  d-\operatorname{Re}\left(  z\right)  -1/2\right)  $, then%
\[
\left\{  \int_{\mathbb{T}^{d}}\left\vert \mathcal{F}_{0}\left(
z,H,R,x\right)  \right\vert ^{p/2}dx\right\}  ^{1/p}\leq C\left(  \frac
{d}{d-1/2-\operatorname{Re}\left(  z\right)  }-p\right)  ^{\left(
d-1/2-\operatorname{Re}\left(  z\right)  \right)  /d-1/2}.
\]
\item[(3)] If $\operatorname{Re}\left(  z\right)  =d-1/2$ and $p<+\infty$, then
\[
\left\{  \int_{\mathbb{T}^{d}}\left\vert \mathcal{F}_{0}\left(
z,H,R,x\right)  \right\vert ^{p/2}dx\right\}  ^{1/p}\leq Cp^{1/2}.
\]
\end{itemize}
\end{lemma}

\begin{proof}
The proof is as in the previous two lemmas. If $p=2$ and $\operatorname{Re}\left(
z\right)  >d/2$ then, by Parseval equality,
\[
\left\{  \int_{\mathbb{T}^{d}}\left\vert \mathcal{F}_{0}\left(
z,H,R,x\right)  \right\vert ^{p/2}dx\right\}  ^{1/p}\leq C\left(
\operatorname{Re}\left(  z\right)  -\dfrac{d}{2}\right)  ^{-1/2}.
\]
If $4\leq p\leq+\infty$ and $\operatorname{Re}\left(  z\right)  >d\left(
1-1/p\right)  -1/2$ then, by the Hausdorff Young inequality,
\begin{equation*}
 \left\{  \int_{\mathbb{T}^{d}}\left\vert \mathcal{F}_{0}\left(
z,H,R,x\right)  \right\vert ^{p/2}dx\right\}  ^{1/p}
  \leq C\left(  \operatorname{Re}\left(  z\right)  -d\left(  1-\dfrac{1}%
{p}\right)  +\dfrac{1}{2}\right)  ^{1/p-1/2}.
\end{equation*}
The cases $2<p<4$ follow from these cases via complex interpolation. By the
above computations,
\begin{align*}
&  \left\{  \int_{\mathbb{T}^{d}}\left\vert \mathcal{F}_{0}\left(
z,H,R,x\right)  \right\vert ^{p/2}dx\right\}  ^{1/p}\\
&  \leq\left\{
\begin{array}
[c]{ll}%
C\varepsilon^{-1/2} & \text{if }p=2\text{ and }\operatorname{Re}\left(
z\right)  =d/2+\varepsilon\text{,}\\
C\varepsilon^{-1/4} & \text{if }p=4\text{ and }\operatorname{Re}\left(
z\right)  =\left(  3d-2\right)  /4+\varepsilon\text{.}%
\end{array}
\right.
\end{align*}
By complex interpolation, with $2<p<4$ and $\operatorname{Re}\left(  z\right)
=d\left(  1-1/p\right)  +2/p-1+\varepsilon$,
\begin{equation*}
  \left\{  \int_{\mathbb{T}^{d}}\left\vert \mathcal{F}_{0}\left(
z,H,R,x\right)  \right\vert ^{p/2}dx\right\}  ^{1/p}
 \leq C\left(  \operatorname{Re}\left(  z\right)  -\left(  d\left(
1-\dfrac{1}{p}\right)  +\dfrac{2}{p}-1\right)  \right)  ^{-1/p}.
\end{equation*}
As before, one can translate the estimates 
 for $\operatorname{Re}(z)\rightarrow \mathrm{critical}\left(  z\right)  +$
into estimates for $p\rightarrow \mathrm{critical}\left(  p\right)  -$.
\end{proof}
\begin{lemma}
\label{L(p)-8} Let $\beta\geq0$ and $j\geq1$, and let $\mathcal{F}_{j}\left(
z,H,R,x\right)  $ be as in the Lemma \ref{Terms}. Then there exists $C$ such
that for every $H,R\geq1$ the following hold.
\begin{itemize}
\item[(1)] If $d/2<\operatorname{Re}\left(  z\right)  \leq\left(  3d-2\right)  /4$
and $2\leq p<(d-2)(d-1-\operatorname{Re}\left(  z\right)  )$, then
\[
\left\{  \int_{\mathbb{T}^{d}}\left\vert \mathcal{F}_{j}\left(
z,H,R,x\right)  \right\vert ^{p/2}dx\right\}  ^{1/p}\leq C\left(  H+R\right)
^{-j}\left(  \dfrac{d-2}{d-1-\operatorname{Re}\left(  z\right)  }-p\right)
^{-\left(  d-1-\operatorname{Re}\left(  z\right)  \right)  /\left(
d-2\right)  }.
\]
\item[(2)] If $\left(  3d-2\right)  /4<\operatorname{Re}\left(  z\right)  <d-1/2$ and
$4\leq p<d/\left(  d-\operatorname{Re}\left(  z\right)  -1/2\right)  $, then%
\[
\left\{  \int_{\mathbb{T}^{d}}\left\vert \mathcal{F}_{j}\left(
z,H,R,x\right)  \right\vert ^{p/2}dx\right\}  ^{1/p}\leq C\left(  H+R\right)
^{-j}\left(  \frac{d}{d-1/2-\operatorname{Re}\left(  z\right)  }-p\right)
^{\left(  d-1/2-\operatorname{Re}\left(  z\right)  \right)  /d-1/2}.
\]
\item[(3)] If $\operatorname{Re}\left(  z\right)  =d-1/2$ and $p<+\infty$, then
\[
\left\{  \int_{\mathbb{T}^{d}}\left\vert \mathcal{F}_{j}\left(
z,H,R,x\right)  \right\vert ^{p/2}dx\right\}  ^{1/p}\leq C\left(  H+R\right)
^{-j}p^{1/2}.
\]
\end{itemize}
\end{lemma}

\begin{proof}
The proof is as in the previous lemmas. If $p=2$ and $\operatorname{Re}\left(
z\right)  >d/2$ then, by Parseval equality,
\begin{align*}
&  \left\{  \int_{\mathbb{T}^{d}}\left\vert \mathcal{F}_{j}\left(
z,H,R,x\right)  \right\vert ^{p/2}dx\right\}  ^{1/p}\\
&  =\left\{  \int_{\mathbb{R}}\left(  R+Hr\right)  ^{-2j}d\mu\left(  r\right)
\sum_{n\in\mathbb{Z}^{d}\setminus\left\{  0\right\}  }\left\vert c\left(  n\right)
\right\vert ^{2}\left\vert n\right\vert ^{-2\operatorname{Re}\left(  z\right)
-2j}\right\}  ^{1/2}\\
&  \leq C\left(  H+R\right)  ^{-j}\left(  \operatorname{Re}\left(  z\right)+j
-\dfrac{d}{2}\right)  ^{-1/2} \leq C\left(  H+R\right)  ^{-j}\left(  \operatorname{Re}\left(  z\right)
-\dfrac{d}{2}\right)  ^{-1/2}.
\end{align*}
If $4\leq p\leq+\infty$ and $\operatorname{Re}\left(  z\right)  >d\left(
1-1/p\right)  -1/2$ then, by the Hausdorff Young inequality,
\begin{align*}
&  \left\{  \int_{\mathbb{T}^{d}}\left\vert \mathcal{F}_{j}\left(
z,H,R,x\right)  \right\vert ^{p/2}dx\right\}  ^{1/p}\\
&  \leq\left\{
{\displaystyle\sum_{k\in\mathbb{Z}^{d}}}
\left\vert \widehat{\mathcal{F}}_{j}\left(  z,H,R,k\right)  \right\vert
^{p/(p-2)}\right\}  ^{\left(  p-2\right)  /2p}\\
&  \leq C\left(  H+R\right)  ^{-j}\left\{
{\displaystyle\sum_{k\in\mathbb{Z}^{d}}}
\left\vert \left(  1+\left\vert k\right\vert \right)  ^{d-1-2\operatorname{Re}%
\left(  z\right)  }\right\vert ^{p/(p-2)}\right\}  ^{\left(  p-2\right)
/2p}\\
&  \leq C\left(  H+R\right)  ^{-j}\left(  \operatorname{Re}\left(  z\right)
-d\left(  1-\dfrac{1}{p}\right)  +\dfrac{1}{2}\right)  ^{1/p-1/2}.
\end{align*}
By complex interpolation, if $2<p<4$ and $\operatorname{Re}\left(  z\right)
=d\left(  1-1/p\right)  +2/p-1+\varepsilon$, then
\begin{equation*}
  \left\{  \int_{\mathbb{T}^{d}}\left\vert \mathcal{F}_{j}\left(
z,H,R,x\right)  \right\vert ^{p/2}dx\right\}  ^{1/p}
  \leq C\left(  H+R\right)  ^{-j}\left(  \operatorname{Re}\left(  z\right)
-\left(  d\left(  1-\dfrac{1}{p}\right)  +\dfrac{2}{p}-1\right)  \right)
^{1/p-1}.
\end{equation*}
As before, one can translate the estimates 
 for $\operatorname{Re}(z)\rightarrow \mathrm{critical}\left(  z\right)  +$
into estimates for $p\rightarrow \mathrm{critical}\left(  p\right)  -$.
\end{proof}

The above lemmas are enough for an upper bound for the norms of the
discrepancy in Theorem \ref{1}(1), for $p$ below the critical index $A$. 
\begin{proof} [Proof of Theorem \ref{1}(1)]
By Lemma \ref{Asymptotic Discrepancy},
\begin{align*}
&  \left\{  {%
{\displaystyle\int_{\mathbb{T}^{d}}}
}\left[  \int_{\mathbb{R}}\left\vert r^{-\left(  d-1\right)  /2}%
\mathcal{D}\left(  r\Omega-x\right)  \right\vert ^{2}d\mu_{H,R}\left(
r\right)  \right]  ^{p/2}dx\right\}  ^{1/p}\\
&  \leq%
{\displaystyle\sum_{j=0}^{h}}
\left\{  {%
{\displaystyle\int_{\mathbb{T}^{d}}}
}\left[  \int_{\mathbb{R}}\left\vert \Phi_{j}\left(  \left(  d+1\right)
/2,r,x\right)  \right\vert ^{2}d\mu_{H,R}\left(  r\right)  \right]
^{p/2}dx\right\}  ^{1/p}\\
&  +\left\{  {%
{\displaystyle\int_{\mathbb{T}^{d}}}
}\left[  \int_{\mathbb{R}}\left\vert \mathcal{R}_{h}\left(  r,x\right)
\right\vert ^{2}d\mu_{H,R}\left(  r\right)  \right]  ^{p/2}dx\right\}  ^{1/p}.
\end{align*}
Since the $\Phi_{j}\left(  \left(  d+1\right)  /2,r,x\right)  $'s are  sums of
two $\Theta_{j}\left(  \left(  d+1\right)  /2,r,x\right)  $'s to which the above
lemmas apply, under appropriate relations between $p$ and $\beta$ the mixed
norm of the discrepancy is uniformly bounded, and (1) follows from the
estimates in Lemma \ref{L(p)-5}, \ref{L(p)-6}, \ref{L(p)-7}, \ref{L(p)-8}.
\end{proof}

 In order to reach the critical index $p=A$ in Theorem \ref{1}(2), one needs an easy
lemma suggested by the Yano extrapolation theorem. See \cite{Yano} or
\cite[Chapter XII-4.41]{Z}.

\begin{lemma} 
\label{Extrapolation 1}
For every function $\mathcal F(x)$ defined on the torus $\mathbb T^d$ the following hold:
 \begin{itemize}
\item[(1)] Let $\alpha\geq0$, $A\geq1$, $K\geq2$, and assume
that
\begin{align*}
\sup_{x\in\mathbb{T}^{d}}\left\{  \left\vert \mathcal{F}\left(  x\right)
\right\vert \right\} & \leq K\text{,}\\
\left\{  \int_{\mathbb{T}^{d}}\left\vert \mathcal{F}\left(  x\right)
\right\vert ^{p}dx\right\}  ^{1/p} &\leq\left(  A-p\right)  ^{-\alpha
}\ \ \text{\ for every }0<p<A\text{.}%
\end{align*}
Then there exists $C$ independent of  $K$ and of $\mathcal{F}\left(  x\right)  $
such that
\[
\left\{ \int_{\mathbb{T}^{d}}\left\vert \mathcal{F}\left(  x\right)
\right\vert ^{A}dx\right\}  ^{1/A}\leq C\log^{\alpha}\left(  K\right)  .
\]
\item[(2)] Assume that $\alpha>0$ and that for every
$p<A<+\infty$,
\[
\left\{  \int_{\mathbb{T}^{d}}\left\vert \mathcal{F}\left(  x\right)
\right\vert ^{p}dx\right\}  ^{1/p}\leq\left(  A-p\right)  ^{-\alpha}.
\]
Then for every $\gamma>1+\alpha A$ there exists $C$ independent of
$\mathcal{F}\left(  x\right)  $ such that
\[
\int_{\mathbb{T}^{d}}\left\vert \mathcal{F}\left(  x\right)  \right\vert
^{A}\log^{-\gamma}\left(  2+\left\vert \mathcal{F}\left(  x\right)
\right\vert \right)  dx\leq C.
\]
\item[(3)] Assume that $\alpha>0$ and that for every $p<+\infty$,
\[
\left\{  \int_{\mathbb{T}^{d}}\left\vert \mathcal{F}\left(  x\right)
\right\vert ^{p}dx\right\}  ^{1/p}\leq p^{\alpha}.
\]
Then for every $\gamma<\alpha/e$ there exists $C>0$ independent of $\mathcal F\left(
x\right)  $ such that
\[
\int_{\mathbb{T}^{d}}\exp\left(  \gamma\left\vert \mathcal{F}\left(  x\right)
\right\vert ^{1/\alpha}\right)  dx\leq C.
\]
\end{itemize}
\end{lemma}

\begin{proof}
(1) If $\alpha\geq0$ and $A\geq1$ and $0<p<A$,
\begin{align*}
&  \left\{  \int_{\mathbb{T}^{d}}\left\vert \mathcal{F}\left(  x\right)
\right\vert ^{A}dx\right\}  ^{1/A}\leq\sup_{x\in\mathbb{T}^{d}}\left\{
\left\vert \mathcal{F}\left(  x\right)  \right\vert ^{\left(  A-p\right)
/A}\right\}  \left\{  \int_{\mathbb{T}^{d}}\left\vert \mathcal{F}\left(
x\right)  \right\vert ^{p}dx\right\}  ^{1/A}\\
&  \leq K^{\left(  A-p\right)  /A}\left(  A-p\right)  ^{-\alpha p/A}%
=A^{-\alpha p/A}\left(  1-p/A\right)  ^{\alpha\left(  1-p/A\right)  }%
K^{1-p/A}\left(  1-p/A\right)  ^{-\alpha}\\
&  \leq K^{1-p/A}\left(  1-p/A\right)  ^{-\alpha}.
\end{align*}
Then, with $1-p/A=t$,
\[
\left\{  \int_{\mathbb{T}^{d}}\left\vert \mathcal{F}\left(  x\right)
\right\vert ^{A}dx\right\}  ^{1/A}\leq\inf_{0<t<1}\left\{  K^{t}t^{-\alpha
}\right\}  =e^{\alpha}\alpha^{-\alpha}\log^{\alpha}\left(  K\right)  .
\] 
(2) Let $\mathcal{F}_{0}\left(  x\right)  =\mathcal{F}\left(  x\right)
\chi_{\left\{  \left\vert \mathcal{F}\left(  x\right)  \right\vert <2\right\}
}\left(  x\right)  $ and $\mathcal{F}_{j}\left(  x\right)  =\mathcal{F}\left(
x\right)  \chi_{\left\{  2^{j}\leq\left\vert \mathcal{F}\left(  x\right)
\right\vert <2^{j+1}\right\}  }\left(  x\right)  $ if $j\geq1$, and let
$\varepsilon_{j}$ be the measure of the set where $\mathcal{F}_{j}\left(
x\right)  \neq0$. Then, if $j\geq1$ and $p<A$,
\[
2^{jp}\varepsilon_{j}\leq%
{\displaystyle\int_{\mathbb{T}^{d}}}
\left\vert \mathcal{F}_{j}\left(  x\right)  \right\vert ^{p}dx\leq%
{\displaystyle\int_{\mathbb{T}^{d}}}
\left\vert \mathcal{F}\left(  x\right)  \right\vert ^{p}dx\leq\left(
A-p\right)  ^{-\alpha p}.
\]
Hence, $\varepsilon_{j}\leq2^{-jp}\left(  A-p\right)  ^{-\alpha p}%
=2^{-Aj}2^{j\left(  A-p\right)  }\left(  A-p\right)  ^{-\alpha A}$, and the
minimum of this expression is attained at $p=A-\alpha A/j\log\left(  2\right)  $.
This gives
\[
\varepsilon_{j}\leq C2^{-Aj}j^{\alpha A}.
\]
Hence, if $\gamma>1+\alpha A$,
\begin{align*}
&  \int_{\mathbb{T}^{d}}\left\vert \mathcal{F}\left(  x\right)  \right\vert
^{A}\log^{-\gamma}\left(  2+\left\vert \mathcal{F}\left(  x\right)
\right\vert \right)  dx\\
&  =%
{\displaystyle\sum_{j=0}^{+\infty}}
{\displaystyle\int_{\mathbb{T}^{d}}}
\left\vert \mathcal{F}_{j}\left(  x\right)  \right\vert ^{A}\log^{-\gamma
}\left(  2+\left\vert \mathcal{F}_{j}\left(  x\right)  \right\vert \right)
dx\\
&  \leq2^{A}\log^{-\gamma}\left(  2\right)  +%
{\displaystyle\sum_{j=1}^{+\infty}}
2^{\left(  j+1\right)  A}\log^{-\gamma}\left(  2+2^{j}\right)  \varepsilon
_{j}\\
&  \leq C+C%
{\displaystyle\sum_{j=1}^{+\infty}}
j^{\alpha A-\gamma}\leq C.
\end{align*}
(3) Let $\mathcal{F}_{j}\left(  x\right)  =\mathcal{F}\left(  x\right)
\chi_{\left\{  j\leq\left\vert \mathcal{F}\left(  x\right)  \right\vert
<j+1\right\}  }\left(  x\right)  $, and let $\varepsilon_{j}$ be the measure of
the set where $\mathcal{F}_{j}\left(  x\right)  \neq0$. Then, if $j\geq1$ and
$p<A$,
\[
j^{p}\varepsilon_{j}\leq%
{\displaystyle\int_{\mathbb{T}^{d}}}
\left\vert \mathcal{F}_{j}\left(  x\right)  \right\vert ^{p}dx\leq%
{\displaystyle\int_{\mathbb{T}^{d}}}
\left\vert \mathcal{F}\left(  x\right)  \right\vert ^{p}dx\leq p^{\alpha p}.
\]
Hence $\varepsilon_{j}\leq j^{-p}p^{\alpha p}$. The minimum of this
expression is attained at $p=e^{-1}j^{1/\alpha}$, and this gives
\[
\varepsilon_{j}\leq\exp\left(  -\left(  \alpha/e\right)  j^{1/\alpha}\right)
.
\]
Hence, if $\gamma<\alpha/e$,
\begin{align*}
&
{\displaystyle\int_{\mathbb{T}^{d}}}
\exp\left(  \gamma\left\vert \mathcal{F}\left(  x\right)  \right\vert
^{1/\alpha}\right)  dx=%
{\displaystyle\sum_{j=0}^{+\infty}}
{\displaystyle\int_{\mathbb{T}^{d}}}
\exp\left(  \gamma\left\vert \mathcal{F}_{j}\left(  x\right)  \right\vert
^{1/\alpha}\right)  dx\\
&  \leq%
{\displaystyle\sum_{j=0}^{+\infty}}
\varepsilon_{j}\exp\left(  \gamma\left(  j+1\right)  ^{1/\alpha}\right)  \leq
e^{\gamma}+%
{\displaystyle\sum_{j=1}^{+\infty}}
\exp\left(  -\left(  \alpha/e-\gamma\left(  1+1/j\right)  ^{1/\alpha}\right)
j^{1/\alpha}\right)  \leq C.
\end{align*}
\end{proof}

\begin{proof}
[Proof of Theorem \ref{1}(2)] 
This follows from part (1) of the theorem via the extrapolation Lemma \ref{Extrapolation 1}. 
In the cases $p<+\infty$
one has just to recall that the
discrepancy satisfy the trivial bound $\left\vert \mathcal{D}\left(
r\Omega-x\right)  \right\vert \leq Cr^{d}$ for every $r\geq1$. The case $d=2$
and $p=+\infty$ and $d\mu\left(  x\right)  =\chi_{\left\{  0<r<1\right\}
}\left(  r\right)  $ is proved in \cite{Huxley1}. An alternative proof of all
cases can also be obtained via the mollified discrepancy defined in Remark \ref{r1}. For example, when
$d=2$, with the techniques in the above lemmas, one can prove that if $1\leq
H\leq R$, and $\delta\leq1/R$,
\begin{align*}
&  \sup_{x\in\mathbb{T}^{2}}\left\{
{\displaystyle\int_{\mathbb{R}}}
\left\vert \left(  r\pm\delta\right)  ^{-3/2}%
{\displaystyle\sum_{n\in\mathbb{Z}^{2}-\left\{  0\right\}  }}
\widehat{\varphi}\left(  \delta n\right)  \widehat{\chi}_{\Omega}\left(
\left(  r\pm\delta\right)  n\right)  \exp\left(  2\pi inx\right)  \right\vert
^{2}d\mu_{H,R}\left(  r\right)  \right\}  ^{1/2}\\
&  \leq\left\{
\begin{array}
[c]{ll}%
C\left\{
{\displaystyle\sum_{n\in\mathbb{Z}^{2}}}
\left(  1+\left\vert \delta n\right\vert \right)  ^{-\gamma}\left(
1+\left\vert k\right\vert \right)  ^{-\beta-1}\right\}  ^{1/2} & \text{if
}0\leq\beta<1\text{,}\\
C\left\{
{\displaystyle\sum_{n\in\mathbb{Z}^{2}}}
\left(  1+\left\vert \delta n\right\vert \right)  ^{-\gamma}\left(
1+\left\vert k\right\vert \right)  ^{-2}\log\left(  2+\left\vert k\right\vert
\right)  \right\}  ^{1/2} & \text{if }\beta=1\text{,}\\
C\left\{
{\displaystyle\sum_{n\in\mathbb{Z}^{2}}}
\left(  1+\left\vert \delta n\right\vert \right)  ^{-\gamma}\left(
1+\left\vert k\right\vert \right)  ^{-2}\right\}  ^{1/2} & \text{if }%
\beta>1\text{,}%
\end{array}
\right. \\
&  \leq\left\{
\begin{array}
[c]{ll}%
C\delta^{\left(  1-\beta\right)  /2} & \text{if }0\leq\beta<1\text{,}\\
C\log\left(  1/\delta\right)  & \text{if }\beta=1\text{,}\\
C\log^{1/2}\left(  1/\delta\right)  & \text{if }\beta>1\text{,}%
\end{array}
\right. \\
&  \leq\left\{
\begin{array}
[c]{ll}%
CR^{\left(  1-\beta\right)  /2} & \text{if }0\leq\beta<1\text{,}\\
C\log\left(  R\right)  & \text{if }\beta=1\text{,}\\
C\log^{1/2}\left(  R\right)  & \text{if }\beta>1\text{.}%
\end{array}
\right.
\end{align*}
\end{proof}

In order to prove the asymptotics of the norms as $H\rightarrow
+\infty$ in Theorem \ref{1} (3), one has to work a bit more. It follows from the previous 
proofs that the main term in the asymptotic expansion of the discrepancy is given by 
\begin{align*}
&  \Phi_{0}\left(  z,r,x\right)  =%
{\displaystyle\sum_{n\in\mathbb{Z}^{d}\setminus\left\{  0\right\}  }}
a_{0}\left(  n\right)  \left\vert n\right\vert ^{-z}\exp\left(  -2\pi
ig\left(  n\right)  r\right)  \exp\left(  2\pi inx\right) \\
& +%
{\displaystyle\sum_{n\in\mathbb{Z}^{d}\setminus\left\{  0\right\}  }}
b_{0}\left(  n\right)  \left\vert n\right\vert ^{-z}\exp\left(  2\pi ig\left(
-n\right)  r\right)  \exp\left(  2\pi inx\right)  .
\end{align*}
\ The following lemma is similar to the previous ones, just observe that one
integrates the square of this function, and not the square of the modulus.

\begin{lemma}
\label{Main Term} Define $\mathcal{G}\left(  z,x\right)  $ and $\mathcal{P}%
\left(  z,H,R,x\right)  $ by
\[
\mathcal{G}\left(  z,x\right)  =\sum_{k\in\mathbb{Z}^{d}}\left(  2\sum
_{\substack{n\in\mathbb{Z}^{d}\setminus\left\{  0,k\right\}  \\ g\left(  n-k\right)  =g\left(
n\right) } }a_{0}\left(  n\right)  b_{0}\left(  k-n\right)  \left\vert
n\right\vert ^{-z}\left\vert k-n\right\vert ^{-z}\right)  \exp\left(  2\pi
ikx\right)  ,
\]%
\[%
{\displaystyle\int_{\mathbb{R}}}
\Phi_{0}\left(  z,r,x\right)  ^{2}d\mu_{H,R}\left(  r\right)  =\mathcal{G}%
\left(  z,x\right)  +\mathcal{P}\left(  z,H,R,x\right)  .
\]
\begin{itemize}
\item[(1)] Under the relations between $p$ and $z$ in Lemma \ref{L(p)-7}, the function 
$\mathcal{G}\left(  z,x\right)  $  is in $L^{p/2}\left(  \mathbb{T}^{d}\right)  $,
\[
\left\{  \int_{\mathbb{T}^{d}}\left\vert \mathcal{G}\left(  z,x\right)
\right\vert ^{p/2}dx\right\}  ^{1/p}\leq C.
\]
\item[(2)] Under the relations between $p$ and $z$ in Lemma \ref{L(p)-5} if $0\leq\beta<1$,
or in Lemma \ref{L(p)-6} if $\beta=1$, or in Lemma \ref{L(p)-7} if $\beta>1$, also the function $\mathcal{P}\left(
z,H,R,x\right)  $ is in $L^{p/2}\left(  \mathbb{T}^{d}\right)  $, and there
exists $C$ such that for every $H,R\geq1$,
\[
\left\{  \int_{\mathbb{T}^{d}}\left\vert \mathcal{P}\left(  z,H,R,x\right)
\right\vert ^{p/2}dx\right\}  ^{1/p}\leq C.
\]
Moreover, if $\beta>0$ then this function vanishes as $H\rightarrow+\infty$,
uniformly in $R\geq1$,
\[
\lim_{H\rightarrow+\infty}\left\{  \int_{\mathbb{T}^{d}}\left\vert
\mathcal{P}\left(  z,H,R,x\right)  \right\vert ^{p/2}dx\right\}  ^{1/p}=0.
\]
\end{itemize}
\end{lemma}

\begin{proof}
Expanding the product $\Phi_{0}\left(  z,r,x\right)  \cdot\Phi_{0}\left(
z,r,x\right)  $ and integrating, one obtains
\begin{align*}
&
{\displaystyle\int_{\mathbb{R}}}
\Phi_{0}\left(  z,r,x\right)  ^{2}d\mu_{H,R}\left(  r\right) \\
&  =\mathcal{G}\left(  z,x\right)  +\mathcal{P}_{1}\left(  z,H,R,x\right)
+\mathcal{P}_{2}\left(  z,H,R,x\right)  +\mathcal{P}_{3}\left(
z,H,R,x\right)  ,
\end{align*}
where
\[
\mathcal{G}\left(  z,x\right)  =2\sum_{k\in\mathbb{Z}^{d}}\sum_{\substack{n\in
\mathbb{Z}^{d}\setminus\left\{  0,k\right\} \\ g\left(  n-k\right)  =g\left(
n\right) } }a_{0}\left(  n\right)  b_{0}\left(  k-n\right)  \left\vert
n\right\vert ^{-z}\left\vert k-n\right\vert ^{-z}\exp\left(  2\pi ikx\right)
,
\]%
\begin{align*}
&  \mathcal{P}_{1}\left(  z,H,R,x\right) \\
&  =2\sum_{k\in\mathbb{Z}^{d}}\sum_{\substack{n\in\mathbb{Z}^{d}\setminus\left\{  0,k\right\}
\\ g\left(  n-k\right)  \neq g\left(  n\right)}  }a_{0}\left(  n\right)
b_{0}\left(  k-n\right)  \left\vert n\right\vert ^{-z}\left\vert
k-n\right\vert ^{-z}\exp\left(  2\pi ikx\right) \\
&  \times\exp\left(  2\pi i\left(  g\left(  n-k\right)  -g\left(  n\right)
\right)  R\right)
{\displaystyle\int_{\mathbb{R}}}
\exp\left(  2\pi iH\left(  g\left(  n-k\right)  -g\left(  n\right)  \right)
r\right)  d\mu\left(  r\right)  ,
\end{align*}%
\begin{align*}
&  \mathcal{P}_{2}\left(  z,H,R,x\right) \\
&  =\sum_{k\in\mathbb{Z}^{d}}\sum_{n\in\mathbb{Z}^{d}\setminus\left\{  0,k\right\}
}a_{0}\left(  n\right)  a_{0}\left(  k-n\right)  \left\vert n\right\vert
^{-z}\left\vert k-n\right\vert ^{-z}\exp\left(  2\pi ikx\right) \\
&  \times\exp\left(  -2\pi i\left(  g\left(  n\right)  +g\left(  k-n\right)
\right)  R\right)
{\displaystyle\int_{\mathbb{R}}}
\exp\left(  -2\pi iH\left(  g\left(  n\right)  +g\left(  k-n\right)  \right)
r\right)  d\mu\left(  r\right)  ,
\end{align*}%
\begin{align*}
&  \mathcal{P}_{3}\left(  z,H,R,x\right) \\
&  =\sum_{k\in\mathbb{Z}^{d}}\sum_{n\in\mathbb{Z}^{d}\setminus\left\{  0,k\right\}
}b_{0}\left(  n\right)  b_{0}\left(  k-n\right)  \left\vert n\right\vert
^{-z}\left\vert k-n\right\vert ^{-z}\exp\left(  2\pi ikx\right) \\
&  \times\exp\left(  2\pi i\left(  g\left(  -n\right)  +g\left(  n-k\right)
\right)  R\right)
{\displaystyle\int_{\mathbb{R}}}
\exp\left(  2\pi iH\left(  g\left(  -n\right)  +g\left(  n-k\right)  \right)
r\right)  d\mu\left(  r\right)  .
\end{align*}
Let us consider the Fourier coefficients of function $\mathcal{G}\left(  z,x\right)$. First observe that these coefficients do not depend on $H$ and $R$.   Since
$a_{0}(n)$ and $b_{0}(-n)$ are bounded, the Fourier coefficient with $k=0$ is
bounded by
\[
\left\vert \widehat{\mathcal{G}}\left(  z,0\right)  \right\vert =\left\vert
2\sum_{n\in\mathbb{Z}^{d}\setminus\left\{  0\right\}  }a_{0}(n)b_{0}(-n)\left\vert
n\right\vert ^{-2z}\right\vert \leq C\sum_{n\in\mathbb{Z}^{d}\setminus\left\{
0\right\}  }\left\vert n\right\vert ^{-2\operatorname{Re}\left(  z\right)
}\leq C.
\]
For an arbitrary $\gamma>1$, the Fourier
coefficients with $k\neq0$ can be bounded by
\begin{align*}
&  \left\vert \widehat{\mathcal{G}}\left(  z,k\right)  \right\vert =\left\vert
2\sum_{\substack{n\in\mathbb{Z}^{d}\setminus\left\{  0,k\right\}  \\ g\left(  n-k\right)
=g\left(  n\right)  }}a_{0}\left(  n\right)  b_{0}\left(  k-n\right)
\left\vert n\right\vert ^{-z}\left\vert k-n\right\vert ^{-z}\right\vert \\
&  \leq C\sum_{n\in\mathbb{Z}^{d}-\left\{  0,k\right\}  }\left\vert
n\right\vert ^{-\operatorname{Re}\left(  z\right)  }\left\vert k-n\right\vert
^{-\operatorname{Re}\left(  z\right)  }\left(  1+\left\vert g\left(
n-k\right)  -g\left(  n\right)  \right\vert \right)  ^{-\gamma}\\
&  \leq C%
{\displaystyle\int_{\mathbb{R}^{d}}}
\left\vert x\right\vert ^{-\operatorname{Re}\left(  z\right)  }\left\vert
k-x\right\vert ^{-\operatorname{Re}\left(  z\right)  }\left(  1+\left\vert
g\left(  x-k\right)  -g\left(  x\right)  \right\vert \right)  ^{-\gamma}dx.
\end{align*}
Hence, by Lemma \ref{Integral}, the last integral is dominated by 
$C\left\vert k\right\vert ^{d-1-2\operatorname{Re}\left(  z\right)  }$.
Therefore, for every $k$, 
\[
\left\vert \widehat{\mathcal{G}}\left(  z,k\right)  \right\vert\leq C(1+\left\vert k\right\vert) ^{d-1-2\operatorname{Re}\left(  z\right)  }.
\]
The estimates of the Fourier coefficients of $\mathcal{P}_{1}\left(
z,H,R,x\right)  $ are similar to the ones of $\mathcal{G}\left(  z,x\right)
$. First observe that $\widehat{\mathcal{P}}_{1}\left(  z,H,R,0\right)  =0$.
Then, by the assumption on the measure $d\mu\left(  r\right)  $, if $k\neq0$
there exists $C$ such that for every $H\geq1$,
\begin{align*}
&  \left\vert \widehat{\mathcal{P}}_{1}\left(  z,H,R,k\right)  \right\vert \\
&  =\Big\vert 2\sum_{\substack{n\in\mathbb{Z}^{d}\setminus\left\{  0,k\right\}  \\ g\left(
n-k\right)  \neq g\left(  n\right)  }}a_{0}\left(  n\right)  b_{0}\left(
k-n\right)  \left\vert n\right\vert ^{-z}\left\vert k-n\right\vert
^{-z} \\
&   \times\exp\left(  2\pi i\left(  g\left(  n-k\right)  -g\left(
n\right)  \right)  R\right)
{\displaystyle\int_{\mathbb{R}}}
\exp\left(  2\pi iH\left(  g\left(  n-k\right)  -g\left(  n\right)  \right)
r\right)  d\mu\left(  r\right) \Big\vert \\
&  \leq C\sum_{n\in\mathbb{Z}^{d}\setminus\left\{  0,k\right\}  }\left\vert
n\right\vert ^{-\operatorname{Re}\left(  z\right)  }\left\vert k-n\right\vert
^{-\operatorname{Re}\left(  z\right)  }\left(  1+H\left\vert g\left(
n-k\right)  -g\left(  n\right)  \right\vert \right)  ^{-\beta}\\
&  \leq C%
{\displaystyle\int_{\mathbb{R}^{d}}}
\left\vert x\right\vert ^{-\operatorname{Re}\left(  z\right)  }\left\vert
k-x\right\vert ^{-\operatorname{Re}\left(  z\right)  }\left(  1+H\left\vert
g\left(  x-k\right)  -g\left(  x\right)  \right\vert \right)  ^{-\beta}dx.
\end{align*}
Hence, by Lemma \ref{Integral}, for every $k\neq 0$,
\[
 \left\vert \widehat{\mathcal{P}}_{1}\left(  z,H,R,k\right)  \right\vert
 \leq\left\{
\begin{array}
[c]{ll}%
C\left\vert k\right\vert ^{d-2\alpha-\beta} & \text{if }0\leq\beta<1\text{,}\\
C\left\vert k\right\vert ^{d-2\alpha-1}\log\left(  2+\left\vert k\right\vert
\right)  & \text{if }\beta=1\text{,}\\
C\left\vert k\right\vert ^{d-2\alpha-1} & \text{if }\beta>1\text{.}%
\end{array}
\right.
\]
These estimates are independent of $H,R\geq1$. Hence, by dominated convergence
applied to the sum that defines $\widehat{\mathcal{P}}_{1}\left(
z,H,R,k\right)  $, if $\beta>0$ then
\[
\lim_{H\rightarrow+\infty}\left\{  \widehat{\mathcal{P}}_{1}\left(
z,H,R,k\right)  \right\}  =0.
\]
The estimates of the Fourier coefficients of $\mathcal{P}_{2}\left(
z,H,R,x\right)  $ and $\mathcal{P}_{3}\left(  z,H,R,x\right)  $ are easier.
Since $g\left(  x\right)  \geq A\left\vert x\right\vert $ with $A>0$,
\begin{align*}
&  \left\vert \widehat{\mathcal{P}}_{2}\left(  z,H,R,k\right)  \right\vert \\
&  =\Big\vert \sum_{n\in\mathbb{Z}^{d}\setminus\left\{  0,k\right\}  }a_{0}\left(
n\right)  a_{0}\left(  k-n\right)  \left\vert n\right\vert ^{-z}\left\vert
k-n\right\vert ^{-z} \\
&   \times\exp\left(  -2\pi i\left(  g\left(  n\right)  +g\left(
k-n\right)  \right)  R\right)
{\displaystyle\int_{\mathbb{R}}}
\exp\left(  -2\pi iH\left(  g\left(  n\right)  +g\left(  k-n\right)  \right)
r\right)  d\mu\left(  r\right)  \Big\vert \\
&  \leq CH^{-\beta}\sum_{n\in\mathbb{Z}^{d}\setminus\left\{  0,k\right\}  }\left\vert
n\right\vert ^{-\operatorname{Re}\left(  z\right)  }\left\vert k-n\right\vert
^{-\operatorname{Re}\left(  z\right)  }\left(  \left\vert n\right\vert
+\left\vert k-n\right\vert \right)  ^{-\beta}\\
&  \leq CH^{-\beta}\left(  1+\left\vert k\right\vert \right)  ^{-\beta}%
\sum_{n\in\mathbb{Z}^{d}\setminus\left\{  0,k\right\}  }\left\vert n\right\vert
^{-\operatorname{Re}\left(  z\right)  }\left\vert k-n\right\vert
^{-\operatorname{Re}\left(  z\right)  }\\
&  \leq CH^{-\beta}\left(  1+\left\vert k\right\vert \right)  ^{d-\beta
-2\operatorname{Re}\left(  z\right)  }.
\end{align*}
Moreover, by this estimate, if $\beta>0$ then
\[
\lim_{H\rightarrow+\infty}\left\{  \widehat{\mathcal{P}}_{2}\left(
z,H,R,k\right)  \right\}  =0.
\]
The estimates of the Fourier coefficients of $\mathcal{P}_{3}\left(
z,H,R,x\right)  $ are analogous to the ones of $\mathcal{P}_{2}\left(
z,H,R,x\right)  $. The estimates of the norms in $L^{p/2}\left(
\mathbb{T}^{d}\right)  $ of these functions in the cases $p=2$ and $p\geq4$
follow from the estimates of the Fourier coefficients of the functions
involved, the Parseval or Hausdorff Young inequality, and dominated
convergence. Finally, the cases $2<p<4$ follow by complex interpolation. The
details are as in the proof of Lemmas \ref{L(p)-5}, \ref{L(p)-6},
\ref{L(p)-7}.
\end{proof}

\begin{proof}
[Proof of Theorem \ref{1}(3)] 
With the
notation of the previous lemmas, since the discrepancy is real one can replace the square
of a modulus with a plain square, and write
\begin{align*}
&  \left\{  \int_{\mathbb{T}^{d}}\left(
{\displaystyle\int_{\mathbb{R}}}
\left\vert r^{-\left(  d-1\right)  /2}\mathcal{D}\left(  r\Omega-x\right)
\right\vert ^{2}d\mu_{H,R}\left(  r\right)  \right)  ^{p/2}dx\right\}
^{1/p}\\
&  =\left\{  \int_{\mathbb{T}^{d}}\left(
{\displaystyle\int_{\mathbb{R}}}
\left(
{\displaystyle\sum_{j=0}^{h}}
\Phi_{j}\left(  \left(  d+1\right)  /2,r,x\right)  +\mathcal{R}_{h}\left(
r,x\right)  \right)  ^{2}d\mu_{H,R}\left(  r\right)  \right)  ^{p/2}%
dx\right\}  ^{1/p}.
\end{align*}
The inner integral is equal to
\begin{align*}
&
{\displaystyle\int_{\mathbb{R}}}
\left(
{\displaystyle\sum_{j=0}^{h}}
\Phi_{j}\left(  \left(  d+1\right)  /2,r,x\right)  +\mathcal{R}_{h}\left(
r,x\right)  \right)  ^{2}d\mu_{H,R}\left(  r\right) \\
&  =\mathcal{G}\left(  \left(  d+1\right)  /2,x\right)  +\mathcal{P}\left(
\left(  d+1\right)  /2,H,R,x\right) \\
&  +%
{\displaystyle\sum_{0\leq i,j\leq h,\ i+j>0}}
{\displaystyle\int_{\mathbb{R}}}
\Phi_{i}\left(  \left(  d+1\right)  /2,r,x\right)  \Phi_{j}\left(  \left(
d+1\right)  /2,r,x\right)  d\mu_{H,R}\left(  r\right) \\
&  +2%
{\displaystyle\sum_{0\leq j\leq h}}
{\displaystyle\int_{\mathbb{R}}}
\Phi_{j}\left(  \left(  d+1\right)  /2,r,x\right)  \mathcal{R}_{h}\left(
r,x\right)  d\mu_{H,R}\left(  r\right) \\
&  +%
{\displaystyle\int_{\mathbb{R}}}
\mathcal{R}_{h}\left(  r,x\right)  ^{2}d\mu_{H,R}\left(  r\right)  .
\end{align*}
By the above lemmas, all these terms give a bounded contribution. The main
term is $\mathcal{G}\left(  \left(  d+1\right)  /2,x\right)  $, and it is
independent of $H$, $R$ and $d\mu(r)$. The contributions of the other terms is negligible
when $\beta>0$ and $H\rightarrow+\infty$. For example, let us estimate the integral with the
mixed product $\Phi_{i}\left(  \left(  d+1\right)  /2,r,x\right)  \Phi
_{j}\left(  \left(  d+1\right)  /2,r,x\right)  $. A repeated application of
the Cauchy Schwarz inequality gives
\begin{align*}
&  \int_{\mathbb{T}^{d}}\left\vert
{\displaystyle\int_{\mathbb{R}}}
\Phi_{i}\left(  \left(  d+1\right)  /2,r,x\right)  \Phi_{j}\left(  \left(
d+1\right)  /2,r,x\right)  d\mu_{H,R}\left(  r\right)  \right\vert ^{p/2}dx\\
&  \leq\int_{\mathbb{T}^{d}}\left(
{\displaystyle\int_{\mathbb{R}}}
\left\vert \Phi_{i}\left(  \left(  d+1\right)  /2,r,x\right)  \right\vert
^{2}d\mu_{H,R}\left(  r\right)  \right)  ^{p/4}\\
&  \times\left(
{\displaystyle\int_{\mathbb{R}}}
\left\vert \Phi_{j}\left(  \left(  d+1\right)  /2,r,x\right)  \right\vert
^{2}d\mu_{H,R}\left(  r\right)  \right)  ^{p/4}dx\\
&  \leq\left\{  \int_{\mathbb{T}^{d}}\left(
{\displaystyle\int_{\mathbb{R}}}
\left\vert \Phi_{i}\left(  \left(  d+1\right)  /2,r,x\right)  \right\vert
^{2}d\mu_{H,R}\left(  r\right)  \right)  ^{p/2}dx\right\}  ^{1/2}\\
&  \times\left\{  \int_{\mathbb{T}^{d}}\left(
{\displaystyle\int_{\mathbb{R}}}
\left\vert \Phi_{j}\left(  \left(  d+1\right)  /2,r,x\right)  \right\vert
^{2}d\mu_{H,R}\left(  r\right)  \right)  ^{p/2}dx\right\}  ^{1/2}.
\end{align*}
By Lemma \ref{L(p)-5}, or Lemma \ref{L(p)-6} or Lemma \ref{L(p)-7}, the terms with $i=0$ or $j=0$ give a bounded contribution. By Lemma \ref{L(p)-8}, the terms with $i>0$ or with $j>0$ converge to $0$ when $H+R\rightarrow
+\infty$.
\end{proof}

\begin{proof}
[Proof of Corollary \ref{1.1}] The corollary is an immediate consequence of part (1)
of the theorem, and of the extrapolation Lemma \ref{Extrapolation 1}. See \cite{Yano} or \cite[Chapter XII-4.41]{Z}.
\end{proof}

\begin{proof}
[Proof of Corollary \ref{1.2}] By the Lemma \ref{Asymptotic Discrepancy},
\[
\mathcal{G}\left(  x\right)  =\sum_{k\in\mathbb{Z}^{d}}\left(  2\sum
_{\substack{n\in\mathbb{Z}^{d}\setminus\left\{  0,k\right\}\\ g\left(  n-k\right)  =g\left(
n\right) } }a_{0}\left(  n\right)  b_{0}\left(  k-n\right)  \left\vert
n\right\vert ^{-\left(  d+1\right)  /2}\left\vert k-n\right\vert ^{-\left(
d+1\right)  /2}\right)  \exp\left(  2\pi ikx\right)  .
\]
Since $A\left\vert x\right\vert \leq g\left(  x\right)  \leq B\left\vert
x\right\vert $, if $g\left(  n-k\right)  =g\left(  n\right)  $ then
$\left\vert k\right\vert \leq C\left\vert n\right\vert $ and
\[
a_{0}\left(  n\right)  b_{0}\left(  k-n\right)  \left\vert n\right\vert
^{-\left(  d+1\right)  /2}\left\vert k-n\right\vert ^{-\left(  d+1\right)
/2}\leq C\left\vert n\right\vert ^{-d-1}\leq C\left\vert k\right\vert
^{-d-1}.
\]
Hence, under the assumption that that for every $m\in\mathbb{Z}^{d}$ the
equation $g\left(  m\right)  =g\left(  n\right)  $ has at most $C$ solutions
in $\mathbb{Z}^{d}$, the Fourier coefficients of $\mathcal{G}\left(
x\right)  $ are bounded by
\[
2\sum_{\substack{n\in\mathbb{Z}^{d}\setminus\left\{  0,k\right\}  \\ g\left(  n-k\right)
=g\left(  n\right) } }a_{0}\left(  n\right)  b_{0}\left(  k-n\right)
\left\vert n\right\vert ^{-\left(  d+1\right)  /2}\left\vert k-n\right\vert
^{-\left(  d+1\right)  /2}\leq C\left\vert k\right\vert ^{-d-1}.
\]
It follows that the Fourier expansion that defines $\mathcal{G}\left(
x\right)  $ is absolutely and uniformly convergent, and this implies that $\mathcal{G}%
\left(  x\right)  $ is bounded and continuous. In particular, under the
additional assumption that $g\left(  m\right)  \neq g\left(  n\right)  $ for
every $m,n\in\mathbb{Z}^{d}$ with $m\neq n$, all Fourier coefficients with
$k\neq0$ vanish, and this function reduces to the constant
\[
2\sum_{n\in\mathbb{Z}^{d}\setminus\left\{  0\right\}  }a_{0}\left(  n\right)
b_{0}\left(  -n\right)  \left\vert n\right\vert ^{-d-1}.
\]
\end{proof}
In order to prove Theorem \ref{2} we need an easy algebraic lemma.

\begin{lemma}
\label{Algebraic} If $\left(  A,B,C,D,...\right)  $ is a vector with integers
coordinates, then the integer vectors $\left(  x,y,z,w,...\right)  $ which are
solutions to the equation $Ax+By+Cz+Dw+...=0$ are a lattice. If $A$
and $B$ are coprimes, so that there exist integers $u$ and $v$
such that $Au+Bv=1$, then a basis of the lattice $\left\{
Ax+By+Cz+Dw+...=0\right\}  $ is
\[
\left\{  \left(  B,-A,0,0,...\right)  ,\left(  uC,vC,-1,0,...\right)  ,\left(
uD,vD,0,-1,...\right)  ,...\right\}  .
\]
The area of a fundamental domain of this lattice is the length of the vector
$\left(  A,B,C,D,...\right)  $,
\[
\sqrt{A^{2}+B^{2}+C^{2}+D^{2}+...}.
\]
\end{lemma}

\begin{proof}
The solutions to the equation $Ax+By+Cz+Dw+...=0$ are a sum of a particular
solution to the non homogeneous equation $Ax+By=-Cz-Dw-...$, plus all
solutions to the homogeneous equation $Ax+By=0$. The solutions to the
homogeneous equation $Ax+By=0$ are $x=Br$ and $y=-Ar$, and a particular
solution to the equation $Ax+By=-Cz-Dw-...$ is $x=-u\left(  Cz+Dw+...\right)
$ and $y=-v\left(  Cz+Dw+...\right)  $. Hence, all integral solutions to
$Ax+By+Cz+Dw+...=0$ are
\[
\left(  x,y,z,w,...\right)  =r\left(  B,-A,0,0,...\right)  +s\left(
uC,vC,-1,0,...\right)  +t\left(  uD,vD,0,-1,...\right)  +....
\]
If $\{\mathbf{e}_{j}\}$ is the standard basis of orthogonal unit vectors, the area of a fundamental domain of the lattice $\left\{
Ax+By+Cz+Dw+...=0\right\}  $ is the length of the vector
\begin{align*}
&  \mathrm{det} \left[
\begin{array}
[c]{ccccc}%
\mathbf{e}_{1} & \mathbf{e}_{2} & \mathbf{e}_{3} & \mathbf{e}_{4} & ...\\
B & -A & 0 & 0 & ...\\
uC & vC & -1 & 0 & ...\\
uD & vD & 0 & -1 & ...\\
... & ... & ... & ... & ...
\end{array}
\right]   \\
&  =\pm A\mathbf{e}_{1}\pm B\mathbf{e}_{2}\pm C\left(  Au+Bv\right)
\mathbf{e}_{3}\pm D\left(  Au+Bv\right)  \mathbf{e}_{4}\pm...\\
&  \mathbf{=}\left(  \mathbf{\pm}A,\mathbf{\pm}B,\mathbf{\pm}C,\mathbf{\pm
}D,...\right)  .
\end{align*}

\end{proof}

\begin{proof}[Proof of Theorem \ref{2}] Let us first prove part (2), that is assume that
$ \widehat{\mu}\left(  \zeta\right)  $ vanishes at infinity. If $\Omega$ is a convex set as in Theorem \ref{1},
set
\[
\mathcal{K}\left(  H,R,x\right)  =%
{\displaystyle\int_{\mathbb{R}}}
\left\vert r^{-\left(  d-1\right)  /2}\mathcal{D}\left(  r\Omega-x\right)
\right\vert ^{2}d\mu_{H,R}\left(  r\right)  .
\]
If the statement of part (2) of the theorem fails, then there exist
$2d/\left(  d-3\right)  <p<+\infty$ and sequences $\left\{  R_{n}\right\}
\rightarrow+\infty$ and $\left\{  H_{n}\right\}  \rightarrow+\infty$ such
that
\[
\limsup_{n\rightarrow+\infty}\left\{  \int_{\mathbb{T}^{d}}\left\vert
\mathcal{K}\left(  H_{n},R_{n},x\right)  \right\vert ^{p/2}dx\right\}
^{2/p}<+\infty.
\]
Then a suitable subsequence converges weakly in $L^{p/2}\left(  \mathbb{T}%
^{d}\right)  $, and this weak convergence implies the convergence of Fourier
coefficients. By the assumption that $\lim_{\left\vert \zeta\right\vert
\rightarrow+\infty}\left\{  \left\vert \widehat{\mu}\left(  \zeta\right)
\right\vert \right\}  =0$ and by Lemma \ref{Terms}, the subsequence 
converges weakly to the function $\mathcal{G}\left( x\right)=\mathcal{G}\left( (d+1)/2, x\right)$ 
defined in Lemma \ref{Main Term},
\[
\mathcal{G}\left(  x\right)  =\sum_{k\in\mathbb{Z}^{d}}\left(  2\sum
_{\substack{n\in\mathbb{Z}^{d}\setminus\left\{  0,k\right\}\\ g\left(  n-k\right)  =g\left(
n\right) } }a_{0}\left(  n\right)  b_{0}\left(  k-n\right)  \left\vert
n\right\vert ^{-\left(  d+1\right)  /2}\left\vert k-n\right\vert ^{-\left(
d+1\right)  /2}\right)  \exp\left(  2\pi ikx\right)  .
\]
Recall that, by Theorem
\ref{1}, this function $\mathcal{G}\left(  x\right)  $ is in $L^{p/2}\left(
\mathbb{T}^{d}\right)  $ for every $p<\left(  2d-4\right)  /\left(
d-3\right)  $. In order to prove the theorem, it suffices to show that when
$\Omega$ is the ball $\Sigma=\left\{  \left\vert x\right\vert \leq1\right\}  $ this function is not
in $L^{p/2}\left(  \mathbb{T}^{d}\right)  $ if $p>2d/\left(  d-3\right)  $. In
order to give an estimate of the norm from below, one can test this function
against a Bessel potential of order $\alpha>0$,
\[
\mathcal{B}\left(  x\right)  =\sum_{k\in\mathbb{Z}^{d}}\left(  1+4\pi
^{2}\left\vert k\right\vert ^{2}\right)  ^{-\alpha/2}\exp\left(  2\pi
ikx\right)  .
\]
This Bessel potential is a positive integrable function, which blows up as
$x\rightarrow0$ with an asymptotic expansion
\[
\mathcal{B}\left(  x\right)  \approx\left\{
\begin{array}
[c]{ll}%
C\left\vert x\right\vert ^{\alpha-d} & \text{if }0<\alpha<d\text{,}\\
C\log\left(  1/\left\vert x\right\vert \right)  & \text{if }\alpha=d\text{,}\\
C & \text{if }\alpha>d\text{.}%
\end{array}
\right.
\]
This follows from the Poisson summation formula, see \cite[Chapter VII.2]{SW}, and the asymptotic estimate
of the Bessel potentials in $\mathbb{R}^{d}$, see \cite[Chapter V 3.1]{Stein1}. 
It follows that if $1\leq r\leq+\infty$ and $\alpha>d\left(  1-1/r\right)  $,
then
\[
\left\{
{\displaystyle\int_{\mathbb{T}^{d}}}
\left\vert \mathcal{B}\left(  x\right)  \right\vert ^{r}dx\right\}
^{1/r}<+\infty.
\]
By the way, when $2\leq r\leq+\infty$ and $1/r+1/s=1$ and $\alpha>d\left(
1-1/r\right)  =d/s$, this also follows via the Hausdorff Young inequality:
\begin{align*}
& \left\{
{\displaystyle\int_{\mathbb{T}^{d}}}
\left\vert \sum_{k\in\mathbb{Z}^{d}}\left(  1+4\pi^{2}\left\vert k\right\vert
^{2}\right)  ^{-\alpha/2}\exp\left(  2\pi ikx\right)  \right\vert
^{r}dx\right\}  ^{1/r}\\
  \leq&\left\{  \sum_{k\in\mathbb{Z}^{d}}\left\vert \left(  1+4\pi
^{2}\left\vert k\right\vert ^{2}\right)  ^{-\alpha/2}\right\vert ^{s}\right\}
^{1/s}<+\infty.
\end{align*}
If $1/r+1/s=1$, then
\begin{align*}
&  2\sum_{k\in\mathbb{Z}^{d}}\left(  1+4\pi^{2}\left\vert k\right\vert
^{2}\right)  ^{-\alpha/2}\sum_{\substack{n\in\mathbb{Z}^{d}\setminus\left\{  0,k\right\}
\\ g\left(  n-k\right)  =g\left(  n\right)}  }a_{0}\left(  n\right)
b_{0}\left(  k-n\right)  \left\vert n\right\vert ^{-\left(  d+1\right)
/2}\left\vert k-n\right\vert ^{-\left(  d+1\right)  /2}\\
&  =%
{\displaystyle\int_{\mathbb{T}^{d}}}
\mathcal{B}\left(  x\right)  \mathcal{G}\left(  x\right)  dx\leq\left\{
{\displaystyle\int_{\mathbb{T}^{d}}}
\left\vert \mathcal{B}\left(  x\right)  \right\vert ^{r}dx\right\}
^{1/r}\left\{
{\displaystyle\int_{\mathbb{T}^{d}}}
\left\vert \mathcal{G}\left(  x\right)  \right\vert ^{s}dx\right\}  ^{1/s}.
\end{align*}
Recall that $g\left(  n\right)  \approx\left\vert n\right\vert $ and that $a_0(n)$ and $b_0(n)$ are positive and bounded from below. 
If $g(n-k)=g(n)$, then $|n|^{-(d+1)/2}|n-k|^{-(d+1)/2}\approx |n|^{-d-1}$. Hence
for every
$\alpha>d\left(  1-1/r\right)  =d/s$ one obtains
\[
\left\{
{\displaystyle\int_{\mathbb{T}^{d}}}
\left\vert \mathcal{G}\left(  x\right)  \right\vert ^{s}dx\right\}  ^{1/s}\geq
C\sum_{k\in\mathbb{Z}^{d}\setminus\left\{  0\right\}  }\left\vert k\right\vert
^{-\alpha}\sum_{\substack{n\in\mathbb{Z}^{d}\setminus\left\{  0\right\}  \\ g\left(  n-k\right)
=g\left(  n\right)}  }\left\vert n\right\vert ^{-d-1}.
\]
Up to this point we have not assumed that the domain is a ball. Now assume that  $\Omega$ is the ball $\Sigma=\left\{  \left\vert x\right\vert \leq1\right\}  $. Then $a_{0}(n)$ and $b_{0}(n)$ are constants and $g\left(  n\right)  =\left\vert n\right\vert $, and the above
inequality takes the more explicit form
\[
\left\{
{\displaystyle\int_{\mathbb{T}^{d}}}
\left\vert \mathcal{G}\left(  x\right)  \right\vert ^{s}dx\right\}  ^{1/s}\geq
C\sum_{k\in\mathbb{Z}^{d}\setminus\left\{  0\right\}  }\left\vert k\right\vert
^{-\alpha}\sum_{\left\vert n-k\right\vert =\left\vert n\right\vert }\left\vert
n\right\vert ^{-d-1}.
\]
In order to bound this expression from below, one can restrict the sum to the
$k$ even,
\[
\sum_{k\in\mathbb{Z}^{d}\setminus\left\{  0\right\}  }\left\vert k\right\vert
^{-\alpha}\sum_{\left\vert n-k\right\vert =\left\vert n\right\vert }\left\vert
n\right\vert ^{-d-1}\geq\sum_{k\in\mathbb{Z}^{d}\setminus\left\{  0\right\}
}\left\vert 2k\right\vert ^{-\alpha}\sum_{\left\vert n-2k\right\vert
=\left\vert n\right\vert }\left\vert n\right\vert ^{-d-1}.
\]
The equation $\left\vert m-2k\right\vert =\left\vert m\right\vert $ is the
same as $k\cdot m=k\cdot k$, and with the change of variables $m=k+n$ one
obtains $k\cdot n=0$, so that for every $\alpha>d/s$,
\[
\left\{
{\displaystyle\int_{\mathbb{T}^{d}}}
\left\vert \mathcal{G}\left(  x\right)  \right\vert ^{s}dx\right\}  ^{1/s}\geq
C\sum_{k\in\mathbb{Z}^{d}\setminus\left\{  0\right\}  }\left\vert k\right\vert
^{-\alpha}\sum_{k\cdot n=0}\left(  \left\vert k\right\vert ^{2}+\left\vert
n\right\vert ^{2}\right)  ^{-\left(  d+1\right)  /2}.
\]
By the above lemma, when two entries of the vector $k$ are coprimes, the area
of a fundamental domain of the $(d-1)$-dimensional lattice $\left\{  k\cdot
n=0\right\}  $ is $\left\vert k\right\vert $, the density of the lattice
is $\left\vert k\right\vert ^{-1}$ and, as a consequence of the classical theorem of Blichfeldt in the geometry of numbers 
(see e.g. \cite[Theorem 9.5]{OLD} for a proof in two dimensions which immediately extends to any dimension), 
for some constant $C$ independent of $k$, one has 
$$\vert\{k\cdot n=0, |n|\leq |k| \} \vert \geq C|k|^{d-2}.$$ Hence, when two entries of the vector $k$ are coprimes,
\[%
{\displaystyle\sum_{k\cdot n=0}}
\left(  \left\vert k\right\vert ^{2}+\left\vert n\right\vert ^{2}\right)
^{-\left(  d+1\right)  /2}\geq\left(  2\left\vert k\right\vert ^{2}\right)
^{-\left(  d+1\right)  /2}\left\vert \left\{  k\cdot n=0,\ \left\vert
n\right\vert \leq\left\vert k\right\vert \right\}  \right\vert \geq
C\left\vert k\right\vert ^{-3}.
\]

By a theorem of E.Ces\`{a}ro, see \cite[Theorem 332]{HW}, the probability that two random non
negative integers are coprime is $6/\pi^{2}$, then the probability that two entries
of the vector $k$ are coprime is positive. 
This implies that if we call $A$ the set of $k\in\mathbb Z^d$ with two coprime entries
and if $\varepsilon$ is sufficiently small and $\eta$ is sufficiently large,
then every shell $\{\eta^j\leq|k|<\eta^{j+1}\}$ contains at least $\varepsilon\eta^{jd}$
integer points in $A$.
Hence, if $\alpha\leq d-3$,
\[
\sum_{k\in\mathbb{Z}^{d}\setminus\left\{  0\right\}  }\left\vert k\right\vert
^{-\alpha}\sum_{k\cdot n=0}\left(  \left\vert k\right\vert ^{2}+\left\vert
n\right\vert ^{2}\right)  ^{-\left(  d+1\right)  /2}\geq C\sum_{k\in
A  }\left\vert k\right\vert ^{-\alpha
-3}=+\infty.
\]
In particular, recalling that $s=p/2$ and $\alpha>d/s=2d/p$, if $p>2d/\left(
d-3\right)  $ then
\[
\left\{  {%
{\displaystyle\int_{\mathbb{T}^{d}}}
}\left\vert \mathcal{G}\left(  x\right)  \right\vert ^{p/2}dx\right\}
^{1/p}=+\infty.
\]
This proves (2). Finally, (1) follows from (2) by replacing the measure
$d\mu\left(  r\right)  $ with a convolution $\varphi\ast\mu\left(  r\right)
dr$, with $\varphi\left(  r\right)  $ a non negative smooth function on
$\mathbb{R}$ with integral one. This convolution is a probability measure with
Fourier transform that vanishes at infinity. Observe that
\begin{align*}
&  \left\{  \int_{\mathbb{T}^{d}}\left(
{\displaystyle\int_{\mathbb{R}}}
\left\vert r^{-\left(  d-1\right)  /2}\mathcal{D}\left(  r\Omega-x\right)
\right\vert ^{2}d\left(  \varphi\ast\mu\right)  _{H,R}\left(  r\right)
\right)  ^{p/2}dx\right\}  ^{2/p}\\
&  =\left\{  \int_{\mathbb{T}^{d}}\left(
{\displaystyle\int_{\mathbb{R}}}
{\displaystyle\int_{\mathbb{R}}}
\left\vert \left(  R+H\left(  r+t\right)  \right)  ^{-\left(  d-1\right)
/2}\mathcal{D}\left(  \left(  R+H\left(  r+t\right)  \right)  \Omega-x\right)
\right\vert ^{2}d\mu\left(  r\right)  \varphi\left(  t\right)  dt\right)
^{p/2}dx\right\}  ^{2/p}\\
&  \leq%
{\displaystyle\int_{\mathbb{R}}}
\left\{  \int_{\mathbb{T}^{d}}\left(
{\displaystyle\int_{\mathbb{R}}}
\left\vert \left(  R+H\left(  r+t\right)  \right)  ^{-\left(  d-1\right)
/2}\mathcal{D}\left(  \left(  R+H\left(  r+t\right)  \right)  \Omega-x\right)
\right\vert ^{2}d\mu\left(  r\right)  \right)  ^{p/2}dx\right\}  ^{2/p}%
\varphi\left(  t\right)  dt.
\end{align*}
Hence, if
\[
\left\{  \int_{\mathbb{T}^{d}}\left(
{\displaystyle\int_{\mathbb{R}}}
\left\vert r^{-\left(  d-1\right)  /2}\mathcal{D}\left(  r\Omega-x\right)
\right\vert ^{2}d\mu_{H,R}\left(  r\right)  \right)  ^{p/2}dx\right\}
^{2/p}\leq C<+\infty,
\]
then also
\[
\left\{  \int_{\mathbb{T}^{d}}\left(
{\displaystyle\int_{\mathbb{R}}}
\left\vert r^{-\left(  d-1\right)  /2}\mathcal{D}\left(  r\Omega-x\right)
\right\vert ^{2}d\left(  \varphi\ast\mu\right)  _{H,R}\left(  r\right)
\right)  ^{p/2}dx\right\}  ^{2/p}\leq C<+\infty,
\]
and the argument used to prove (2) applies.
\end{proof}

\section{Concluding remarks}

\begin{remark}
\label{ellissoidi}
The following examples show that Corollary \ref{1.2} 
is non void.
The ellipsoid $\Omega=\left\{  \left\vert M\left(  x-p\right)  \right\vert
\leq1\right\}  $, with $M$ a non singular $d\times d$ matrix and $p$ a fixed
point in $\mathbb{R}^{d}$, has support function $g\left(  x\right)
=\left\vert \left(  M^{t}\right)  ^{-1}x\right\vert +x\cdot p$,
\begin{gather*}
g\left(  x\right)  =\sup_{y\in\Omega}\left\{  x\cdot y\right\}  =\sup
_{\left\{  \left\vert M\left(  y-p\right)  \right\vert \leq1\right\}
}\left\{  M^{t}\left(  M^{t}\right)  ^{-1}x\cdot y\right\} \\
=\sup_{\left\{  \left\vert M\left(  y-p\right)  \right\vert \leq1\right\}
}\left\{  \left(  M^{t}\right)  ^{-1}x\cdot M\left(  y-p\right)  \right\}
+\left(  M^{t}\right)  ^{-1}x\cdot Mp=\left\vert \left(  M^{t}\right)
^{-1}x\right\vert +x\cdot p.
\end{gather*}
Then the equality $g\left(  m\right)  =g\left(  n\right)  $ gives a non trivial
algebraic relation between the coordinates of $p=\left(  p_{1},p_{2}%
,...,p_{d}\right)  $,
\[
\left(  m-n\right)  \cdot p=\left\vert \left(  M^{t}\right)  ^{-1}n\right\vert
-\left\vert \left(  M^{t}\right)  ^{-1}m\right\vert .
\]

If $\left\{  1,p_{1},p_{2},...,p_{d}\right\}  $ are linearly independent over
the algebraic closure of the field generated by the entries of the matrix $M$,
then this relation holds only if $m=n$. Hence, under these assumptions, the
support function is injective when restricted to the integers. 

In the case
$p=0$, then the equality $g\left(  m\right)  =g\left(  n\right)  $ when
squared gives
\begin{gather*}
Am_{1}^{2}+Bm_{2}^{2}+...+Cm_{1}m_{2}+...=An_{1}^{2}+Bn_{2}^{2}+...+Cn_{1}%
n_{2}+...,\\
A\left(  m_{1}^{2}-n_{1}^{2}\right)  +B\left(  m_{2}^{2}-n_{2}^{2}\right)
+...+C\left(  m_{1}m_{2}-n_{1}n_{2}\right)  +...=0.
\end{gather*}
Here $m=\left(  m_{1},m_{2},...\right)  $, $n=\left(  n_{1},n_{2},...\right)
$, and $A$, $B$, $C$,... are homogeneous second degree polynomials in the
entries of the matrix $M$. This equation has at least the solutions $n=\pm
m$. On the other hand, if the entries of the matrix $M$ are algebraically
independent, then $n=\pm m$ are the only solutions. A cardinality argument
shows that for a fixed matrix $M$, then almost every point $p$ has the property that there
exists no algebraic relation between its coordinates. Similarly, almost every
matrix $M$ has the property that its entries are algebraically independent.
\end{remark}

\begin{remark}
In Corollary \ref{1.2} we defined a convex set \textit{\textquotedblleft
generic\textquotedblright} if its support function is injective when
restricted to the integers. By the above remark these generic convex sets exist. Moreover, they are the majority, they
are of second category in space of compact convex sets endowed with the
Hausdorff metric. If $A+\Omega$ is the Minkowski sum of $A$ and $\Omega$, then
$g_{rA+\Omega}\left(  x\right)  =rg_{A}\left(  x\right)  +g_{\Omega}\left(
x\right)  $. For a fixed $x$ in $\mathbb{R}^{d}$, the function $\Omega
\rightarrow g_{\Omega}\left(  x\right)  $ is continuous in the Hausdorff
metric. For fixed $m,n\in\mathbb{Z}^{d}$ with $m\neq n$, the collection of
convex sets $\Omega$ with $g_{\Omega}\left(  m\right)  \neq g_{\Omega}\left(
n\right)  $ is open in the Hausdorff metric. On the other hand, if $g_{\Omega
}\left(  m\right)  =g_{\Omega}\left(  n\right)  $, and if $A$ is a convex set
with $g_{A}\left(  m\right)  \neq g_{A}\left(  n\right)  $, as in the previous
remark, then $rA+\Omega\rightarrow\Omega$ as $r\rightarrow0+$, and
$g_{rA+\Omega}\left(  m\right)  \neq g_{rA+\Omega}\left(  n\right)  $. This
implies that the set of $\Omega$ with $g_{\Omega}\left(  m\right)  \neq
g_{\Omega}\left(  n\right)  $ is open and dense. Hence the set of $\Omega$
with $g_{\Omega}\left(  m\right)  \neq g_{\Omega}\left(  n\right)  $ for every
$m,n\in\mathbb{Z}^{d}$ with $m\neq n$ is the intersection of a countable
family of open dense sets.
\end{remark}

\begin{remark}
For the ball centered at the origin $\Sigma=\left\{  \left\vert x\right\vert
\leq1\right\}  $ the function $\mathcal{G}\left(  x\right)  $ defined in Lemma
\ref{Main Term} and studied in Theorem \ref{2} is not constant. On the other hand,
by Remark \ref{ellissoidi}, for almost
every $p$ the function $\mathcal{G}\left(  x\right)  $ associated to the
shifted ball $\Omega=\left\{  \left\vert x-p\right\vert \leq1\right\}
=\Sigma+p$ is constant. This may seem contradictory, but observe that in the
case of $\Sigma$ the function $\mathcal{G}\left(  x\right)  $ is an $r$
average of the discrepancy of $r\Sigma-x$, while if $\Omega=\Sigma+p$ the
function $\mathcal{G}\left(  x\right)  $ is an $r$ average of the discrepancy
of $r\Omega-x=r\Sigma+\left(  rp-x\right)  $. These averages are different. In
particular, in the averages of $r\Sigma+\left(  rp-x\right)  $ are a mix of an
average over the dilations $r\Sigma$ together with an average over the
translations $rp-x$. Observe that for irrational choices of $p$, these
translations $rp-x$ are dense in the set of all translations. Hence it is not
completely surprising that in this case $\mathcal{G}\left(  x\right)  $ is constant.
\end{remark}

\bibliographystyle{amsplain}

\end{document}